\documentclass[11pt]{article}

\usepackage[T1]{fontenc}
\usepackage[utf8]{inputenc}
\usepackage{lmodern}
\usepackage[margin=1in]{geometry}
\usepackage{amsmath,amssymb,amsthm,mathtools}
\usepackage{fullpage}
\usepackage{authblk}
\usepackage{booktabs}
\usepackage{enumitem}
\usepackage{graphicx}
\usepackage{tikz}
\usepackage{pifont}
\usepackage{subcaption}
\usepackage[numbers,sort&compress]{natbib}
\usepackage{xcolor}
\usepackage[hidelinks]{hyperref}
\hypersetup{
    pdftitle={Almost Sure Convergence of Nonlinear Stochastic Approximation: An Interplay of Noise and Step Size},
    pdfauthor={Hoang Huy Nguyen, Quang Nguyen, Anh Duc Nguyen, Siva Theja Maguluri},
    pdfkeywords={Stochastic Approximation, Almost Sure Convergence, Heavy-Tailed Noise, High-Moment Noise}
}

\allowdisplaybreaks
\newtheorem{theorem}{Theorem}[section]
\newtheorem{lemma}[theorem]{Lemma}
\newtheorem{proposition}[theorem]{Proposition}
\newtheorem{corollary}{Corollary}[section]

\newtheorem{assumption}{Assumption}
\newtheorem{remark}{Remark}[section]
\newtheorem{claim}{Claim}
\newtheorem{case}{Case}

\newcommand{\E}{\mathbb{E}}
\renewcommand{\P}{\mathbb{P}}
\newcommand{\R}{\mathbb{R}}
\newcommand{\Z}{\mathbb{Z}}
\newcommand{\cD}{\mathcal{D}}
\newcommand{\cF}{\mathcal{F}}
\newcommand{\cX}{\mathcal{X}}
\newcommand{\norm}[1]{\left\lVert#1\right\rVert}
\newcommand{\br}[1]{\left(#1\right)}
\newcommand{\sqbr}[1]{\left[#1\right]}
\newcommand{\xmark}{\text{\ding{55}}}

\title{Almost Sure Convergence of Nonlinear Stochastic Approximation:\\
An Interplay of Noise and Step Size}
\author[1]{Hoang Huy Nguyen\thanks{\texttt{hnguyen455@gatech.edu}}}
\author[2]{Quang Nguyen\thanks{\texttt{23521285@gm.uit.edu.vn}}}
\author[3]{Anh Duc Nguyen\thanks{\texttt{ducna@nus.edu.sg}}}
\author[1]{Siva Theja Maguluri\thanks{\texttt{siva.theja@gatech.edu}}}
\affil[1]{H. Milton Stewart School of Industrial and Systems Engineering, Georgia Institute of Technology, Atlanta, Georgia, USA}
\affil[2]{Department of Computer Science, University of Information Technology,\\
Vietnam National University, Ho Chi Minh City, Vietnam}
\affil[3]{Institute of Operations Research and Analytics, National University of Singapore, Singapore}

\date{}

\begin{document}
\maketitle

\begin{abstract}
We study the almost sure convergence of the Stochastic Approximation algorithm to the fixed point $x^\star$ of a nonlinear operator under a negative drift condition and a general noise sequence with finite $p$-th moment for some $p > 1$. Classical almost sure convergence results of Stochastic Approximation are mostly analyzed for the square-integrable noise setting, and it is shown that any non-summable but square-summable step size sequence is sufficient to obtain almost sure convergence. However, such a limitation prevents wider algorithmic application. In particular, many applications in Machine Learning and Operations Research admit heavy-tailed noise with infinite variance, rendering such guarantees inapplicable. On the other hand, when a stronger condition on the noise is available, such guarantees on the step size would be too conservative, as practitioners would like to pick a larger step size for a more preferable convergence behavior. To this end, we show that any non-summable but $p$-th power summable step size sequence is sufficient to guarantee almost sure convergence, covering the gap in the literature.

Our guarantees are obtained using a universal Lyapunov drift argument. For the regime $p \in (1, 2)$, we show that using the $\norm{x-x^\star}^p$-like Lyapunov function and applying the Holder-smoothness bound suffice. For $p > 2$, such an approach is no longer applicable, and therefore, we introduce a novel iterate projection technique to control the nonlinear terms produced by high-moment bounds and multiplicative noise.  We believe our proof techniques and their implications could be of independent interest and pave the way for the analysis of Stochastic Approximation under a general noise and a general drift condition. As a starting point, we further establish almost sure convergence results for non-expansive operators under a multiplicative noise condition for $p \in (1,2]$ and an additive noise condition for $p \in (2,\infty)$.
\end{abstract}

\section{Introduction}
\label{sec:intro}

Many problems in modern Machine Learning (ML) and Operations Research (OR) can be cast as a fixed-point equation problem for a certain operator with a noisy oracle \citep{bhandari-td-learning, qu-wierman-async, srikant-ying-td-learning, zaiwei-envelope, zaiwei-triad, sajad-federated-rl, hoang-nonlinear-sa, zaiwei-stochastic-game-2023, asymptotic-variance-shubhada-agrawal24a, Mertikopoulos2024-volkan-stochastic-approximation-games}. Under this viewpoint, a common approach to solve such fixed-point equations (FPEs) is to use the Stochastic Approximation algorithm, which was first introduced in \cite{Robbins&Monro:1951}. The focus of this work is to establish an almost-sure convergence guarantee for the Stochastic Approximation algorithm, which can be written as the following stochastic, discrete-time update step
\begin{align}
    \label{equation: main_iterate}
    x_{k+1} = x_k + \alpha_k(H(x_k)-x_k+w_k).
\end{align}
Here, $H$ is the operator of interest that we want to find the fixed-point solution, $\alpha_k$ is the step size and $w_k$ is a random variable that represents the noise. Classical almost sure convergence results for Stochastic Approximation typically consider the setting where the noise is unbiased and square-integrable \citep{Borkar2008StochasticAA}. In this setting, almost-sure convergence can be shown as long as the step-sizes are non-summable, but square summable, i.e., $\sum \alpha_k = \infty, \sum \alpha_k^2 < \infty$ \citep{Borkar2008StochasticAA, zaiwei-envelope, hoang-nonlinear-sa}. This choice of step-sizes is in fact tight, as there are known counterexamples that show almost-sure convergence is not possible otherwise (for example, see Appendix \ref{ssec: impossibility-theorem-proof}). On the other hand, it is unclear how would the algorithm behave beyond the classical square-integrable assumption. Therefore, our goal is to study the Stochastic Approximation iterates under a general noise condition (e.g., the noise has a finite $p$-th moment for some $p > 1$) and characterize the set of step sizes where the iterates can converge almost surely to the fixed point of $H$.

In particular, many applications in Machine Learning and Applied Probability admit heavy-tailed noise with infinite variance \citep{eytan-heavy-tailed, tail-index-simsekli19a, infinite-variance-sgd-mert, ziv-heavy-tailed-jobs, general-purpose-shaan}. While the noise in such settings have an unbounded second moment, it is possible that its $p$-th moment is finite for some $1 < p < 2$. In this case, we show that one still obtains almost sure convergence by allowing a smaller choice of step sizes than in the classical setting.

On the other hand, one may have a much stronger handle on the noise beyond the second moment in some settings, i.e., the $p$-th moment of the noise exists for some $p > 2$. In this case, we show that one can pick larger step sizes while keeping almost sure convergence. This is preferred in practice because larger step sizes typically lead to faster convergence at the beginning \citep{Goodfellow2016DeepLearning,yu-constant-stepsize, zaiwei-constant-stepsize}. Hence, our results provide a more fine-grained trade-off between noise and step sizes, which leads to wider algorithmic implications.

\subsection{Contributions}
\label{ssec: contribution}

In this work, we consider the Stochastic Approximation algorithm of a  "nice enough" operator $H$. We characterize the niceness of the operator in terms of a general drift condition (see Assumption \ref{assumption: general-drift-condition}), and our setting encompasses many popular Stochastic Approximation settings such as Hurwitz linear operators, dissipativity, or contractive operators. We consider a general multiplicative unbiased martingale difference noise sequence whose magnitude can scale with the iterates. Moreover, we assume that the $p$-th moment is finite for some $p>1$. Then, we show in Theorem \ref{theorem: main-result} that one has almost sure asymptotic convergence to the fixed point $x^\star$ of $H$ as long as the diminishing step size sequence $(\alpha_k)_{k \geq 0}$ satisfies the condition,
\vspace{-2pt}
\begin{align}
    \label{eqn: step-size-condition}
    \sum \alpha_k = \infty, \qquad \sum \alpha_k^p < \infty.
\end{align}
In particular, if one uses step sizes of the form $\alpha_k=\alpha (k+K)^{-\xi}$, then the above condition imposes that $\xi \in (1/p,1]$. We show that this condition is tight in the following sense. We present a counterexample in Theorem \ref{thm:tight} where picking step sizes with $\xi\in (0,1/p]$ does not lead to almost sure convergence. When one has $p=2$, one recovers the classical results. Another special case of our result is the Strong Law of Large Numbers (SLLN) using step sizes $\alpha_k = \frac{1}{k+1}$, i.e. $\xi = 1$. Thus, our result provides a generalization of SLLN where one is allowed to use more general step sizes depending on the moment condition of the random variables.

We obtain our results by considering the regimes $p \in (1,2]$ and $p \in (2,\infty)$ as two separate cases. First, when $1 < p \leq 2$, we directly apply the Almost Supermartingale Convergence Theorem \citep{robbins1971convergence,neurodynamic} using a Lyapunov/potential function that behaves similarly to $\norm{x - x^\star}^p$. This potential function enables us to generalize the classical proof of $p=2$ in a straightforward manner. Then, in the case of $p \in (2,\infty)$, it is natural to again use the Lyapunov function, $\norm{x - x^\star}^p$. However, this setting is more challenging, and so, one cannot directly obtain a tight enough drift to recover the step size condition \eqref{eqn: step-size-condition} due to additional nonlinear cross-terms between $\norm{x-x^\star}$ and the noise. To this end, we propose the "iterate projection method" by constructing a proxy variable that can be viewed as a projection of $\norm{x_k-x^\star}$ onto the bounded interval that we will collapse to $0$. Thanks to the boundedness of $z_k$, we can directly prove the almost sure convergence of $z_k$ using the step size conditions \eqref{eqn: step-size-condition}, and then use a bootstrap argument to eventually show almost sure convergence of $x_k$.

Compared to previous works on the $p \in (2,\infty)$ regime \citep{mertikopoulos2020almost, jin2025stochastic}, our projection technique allows us to handle multiplicative noise and lift the global/local $2p-2$-th moment assumption posed in these works. Moreover, our approach is strong enough to cover non-expansive operators (Theorem \ref{theorem: nonexpansive-a.s.-convergence}), which has the potential to be further generalized to cover the nonconvex optimization setting studied in these works as well (see Lemma \ref{lemma: H(x)-x bound when x is far from x*} and its remark).

\subsection{Literature overview}
\label{ssec: literature}

\begin{table}[htbp]
\centering
\footnotesize
\renewcommand{\arraystretch}{1.3}
\resizebox{\textwidth}{!}{
\begin{tabular}{@{}l c c l c@{}}
\toprule
\textbf{Reference} & $\boldsymbol{p}$ & \textbf{Other condition} & \textbf{Noise Cond.} & Drift condition \\\hline

\citep{zaiwei-envelope} & $\{2\}$ & \xmark & Multiplicative & Contractivity \\
\citep{caio-meyn2025revisitingstepsizeassumptionsstochastic} & $\{\infty\}$ & DV3 and other growth conditions & Markovian (Mult.) & ODE global stability \\
\textbf{This Work (Theorem \ref{theorem: main-result})} & $\boldsymbol{(1,\infty)}$ & \textbf{\xmark} & \textbf{Multiplicative} & \textbf{Assumption \ref{assumption: general-drift-condition}} \\ \bottomrule
\citep{krasulina1969infinite-variance}\textsuperscript{$\dagger$} & $(1, 2]$ & \xmark & Additive & Mild drift condition\textsuperscript{$\ddagger$} \\
\citep{mertikopoulos2020almost} & $[2,\infty)$ & Global $2p-2$-th moment & Additive & Bounded $\nabla f$ near $x^\star$\textsuperscript{$\ddagger$} \\
\citep{jin2025stochastic} & $[2,\infty)$ & Local $2p-2$-th moment & Weak-growth & Bounded $\nabla f$ near $x^\star$\textsuperscript{$\ddagger$} \\
\citep{bravo-cominetti-non-expansive} & $\{2\}$ & \xmark & Multiplicative & Non-expansive \\
\citep{blaser2026asymptotic-non-expansive} & $\{2\}$ & \xmark & Markovian (Add.) & Non-expansive \\
\textbf{This Work (Theorem \ref{theorem: nonexpansive-a.s.-convergence})} & $\boldsymbol{(1,2]}$ & \textbf{\xmark} & \textbf{Multiplicative} & \textbf{Non-expansive} \\
\textbf{This Work (Theorem \ref{theorem: nonexpansive-a.s.-convergence})} & $\boldsymbol{(2,\infty)}$ & \textbf{\xmark} & \textbf{Additive} & \textbf{Non-expansive} \\
\bottomrule
\multicolumn{5}{@{}l}{\small \textsuperscript{$\dagger$} This result is only applicable for scalar iterates.}\\
\multicolumn{5}{@{}l}{\small \textsuperscript{$\ddagger$} This assumption is similar to our Lemma \ref{lemma: H(x)-x bound when x is far from x*}, which can be obtained via non-expansivity.}
\end{tabular}
}
\caption{Comparison of Almost Sure Convergence Results in Stochastic Approximation under a drift assumption similar to Assumption \ref{assumption: general-drift-condition} (such as contractivity, dissipativity, etc). We mention works with Markovian noise for completeness although it is out of our scope.}
\label{tab:sa_comparison_full}
\end{table}

\textbf{Stochastic Approximation}: The Stochastic Approximation algorithm was first proposed by \citep{Robbins&Monro:1951} and since then, researchers have found many applications in Reinforcement Learning \citep{bhandari-td-learning,qu-wierman-async,sajad-two-time-scale-tac,sajad-federated-rl, zaiwei-triad, chandak2025finite,durmus2025finite,Chen2025-sa-concentration}, Game Theory \citep{zaiwei-stochastic-game-2023, Mertikopoulos2024-volkan-stochastic-approximation-games}, Control \citep{hanfuchen-adaptive-regulator, HanFuChenSAApplication, hoang-nonlinear-sa} and Statistics \citep{asymptotic-variance-shubhada-agrawal24a}. Thus, theoretically understanding the behavior of the Stochastic Approximation iterates under various noise and step size settings has been a major topic in the ML and OR communities. In particular, many works have explored its behavior under the constant step size setting (i.e. $\xi = 0$) \citep{zhang2024prelimit,huo2024collusion-constant-sa,zaiwei-constant-stepsize,yu-constant-stepsize} due to its desirable convergence behavior in practice \citep{Goodfellow2016DeepLearning}. However, it would be impossible for the Stochastic Approximation iterates to converge due to the influence from the noise, and so various works have studied diminishing step sizes (i.e. $\xi \in (0,1]$) \citep{zaiwei-envelope, sajad-two-time-scale-tac, hoang-nonlinear-sa} where finite-time mean squared (m.s.) convergence can be obtained. On the other hand, several other works instead study almost sure (a.s.) convergence \citep{blum-almost-sure,Borkar2008StochasticAA,kushner-yin-book,mertikopoulos2020almost,karandikar2024convergence,liu2024almost,weissmann2024almost,jin2025stochastic}. In particular, several previous works such as \citep{Borkar2008StochasticAA,kushner-yin-book, kontoyiannisborkar2024odemethodasymptoticstatistics, jin2024stabilityconvergenceanalysis-adagrad,jin2025stochastic,caio-meyn2025revisitingstepsizeassumptionsstochastic} establish almost sure convergence via the so-called ODE method by investigating continuous-time dynamics. When the noise is square-integrable, \citep{Borkar2008StochasticAA, kushner-yin-book} show the behavior of the discrete-time system \eqref{equation: main_iterate} is identical to that of the ODE whenever the step sizes are square-summable, in contrast to the fact that m.s. convergence can be obtained for any non-summable diminishing step sizes.

\textbf{Law of Large Numbers}: In the special case that $H \equiv x^\star, \alpha_k = \frac{1}{k+1}$, one can show that proving almost sure convergence for the Stochastic Approximation iterates is equivalent to the SLLN \citep{durrett2019probability}. Under this viewpoint, there have been many proofs using several different approaches, such as using maximal inequality \citep{J_Michael_Steele2015-maximal-lln}, random walk \citep{Curien2022-random-walk-slln-proof}, monotonic convergence from subsequence using second moment assumption \citep{grimmett-stirzaker-book, gallager1996book, bremaud2020markov-book}, Borel-Cantelli on the fourth moment assumption \citep{gallager1996book, bremaud2020markov-book, durrett2019probability}, and truncation \citep{Etemadi1981-slln, gallager1996book, bremaud2020markov-book, durrett2019probability}. However, these approaches either rely on making strong assumptions or computing the next update using all previous iterates, which are hardly applicable for the Nonlinear Stochastic Approximation setting.

\textbf{Noise assumption}: Various works have studied the convergence of the Stochastic Approximation beyond the classical square-integrability assumption of the noise. For the heavy-tailed regime $p \in (1,2)$, \citep{krasulina1969infinite-variance, goodsell1976almost, li1993almost} studies the scalar Robbins--Monro recursion with additive noise, and to the best of our knowledge, there is no work that obtains almost sure convergence for a general Nonlinear Stochastic Approximation algorithm. Besides this, \citep{infinite-variance-sgd-mert,wang2021convergence, fatkhullin2025can} obtains finite-time convergence for SGD in the infinite variance noise setting. For the regime $p > 2$, \citep{mertikopoulos2020almost} and \citep{jin2025stochastic} studied the almost sure convergence of SGD when the $p$-th moment of the noise is bounded. However, they only considered additive noise or noise with a weak-growth condition, and both papers additionally require the stringent bounded $(2p-2)$-th moment noise assumption. The recent work \citep{caio-meyn2025revisitingstepsizeassumptionsstochastic} presents almost sure convergence for all choices of step-sizes of the form $\alpha_k=\alpha/k^\xi$ for $\xi\in(0,1)$. However, their results are under a strong set of assumptions \citep{kontoyiannis2005largedeviation}, which ensure that the noise has all moments.

In contrast, our work provides the complete trade-off between noise and step-sizes by showing that when the $p$-th moment of noise is finite, one can use step-sizes that satisfy \eqref{eqn: step-size-condition}. For a comparison between our results with other works in the literature, we refer the readers to Table \ref{tab:sa_comparison_full}.

\section{Problem setting and Main Results}
\label{sec: main-results}

\subsection{Problem setting}

Let $H$ be a (possibly non-linear) operator $H: \mathbb{R}^d \rightarrow \mathbb{R}^d$. Our main interest is to solve:
\begin{align}
    \label{eqn: fixed-point-equation}
    H(x) = x
\end{align}
where we only have access to a noisy oracle. In contrast to classical deterministic FPE solvers that assume access to the exact oracle $H(x)$, real-life applications usually revolve around approximations of the form $H(x) + w$, where $w$ represents the inexactness. Thus, to efficiently solve~\eqref{eqn: fixed-point-equation}, the stochastic iterative algorithm~\eqref{equation: main_iterate} is usually employed with a suitable step size $\alpha_k$. The choice of step size is crucial to ensure the robustness and theoretical advantages of the algorithm. In this paper, we focus on analyzing SA with a decreasing step size $\alpha_k$, which is usually chosen as $\alpha_k = \frac{\alpha}{(k+K)^{\xi}}$, where $\xi \in (0,1]$, and determine sufficient and necessary conditions to establish almost sure convergence. Thus, we make the following assumptions that are motivated by applications in Reinforcement Learning and Control~\cite{nonlinear-sa, zaiwei-envelope, inverse-rl-passive-langevin, hoang-nonlinear-sa}. The first assumption imposes a general negative-drift Lyapunov condition that is satisfied for an appropriate potential function that is smooth and has quadratic growth:

\begin{assumption}[Negative drift condition]
\label{assumption: general-drift-condition}
Let $\Phi : \mathbb{R}^d \to \mathbb{R}_+$ be a Lyapunov function. There exist constants $0 < c_1 \leq c_2, 0 < \eta$ and $0 < L_2$ such that
\begin{align}
&\big\langle \nabla \Phi(x - x^\star), H(x) - x \big\rangle
\le - \eta \, \Phi(x - x^\star),
\qquad \forall x \in \mathbb{R}^d,
\label{eqn: negative-drift}
\\[0.6em]
\Phi(y)
&\le \Phi(x)
+ \big\langle \nabla \Phi(x), y - x \big\rangle
+ \frac{L_2}{2} \, \| y - x \|_2^{2},
\qquad \forall x,y \in \mathbb{R}^d,
\label{eqn: holder-smoothness}
\\[0.6em]
&c_1 \, \|x-x^\star\|_2^{2}
\le \Phi(x-x^\star)
\le c_2 \, \|x-x^\star\|_2^{2},
\qquad \forall x \in \mathbb{R}^d,
\label{eqn: lyapunov-sandwich}
\end{align}
\end{assumption}

This general negative drift condition was previously studied in \cite{general-purpose-shaan}, and it encompasses many Stochastic Approximation settings such as Hurwitz linear operators \cite{srikant-ying-td-learning, haque2025tightfinitetimebounds}, contractive operators \cite{zaiwei-envelope}, dissipativity \cite{nonlinear-sa}, and exponential stability \cite{hoang-nonlinear-sa}. In addition to Stochastic Approximation, this negative drift condition also appears in other applications in Markov chain mixing or learning \cite{raginsky2017nonconvexsgld, taghvaei-lyapunov-poincare, hoang-erlang-c-mixing}. At first glance, this assumption seems restrictive, as these inequalities combined seem to suggest that only quadratic functions (such as $\norm{x-x^\star}^2$) are viable Lyapunov functions. In practice, this assumption is far more flexible and can cover a large class of Lyapunov functions. We shall provide detailed justifications for each inequality in Assumption \ref{assumption: general-drift-condition} as follows.

\textbf{Justification for the polynomial growth \eqref{eqn: lyapunov-sandwich}:} To control the growth of the Lyapunov function $\Phi$ and to control the distance to the fixed points, we additionally assume that $\Phi$ has a polynomial growth condition. Suppose that there exist constants $c_{1,\omega}, c_{2,\omega}, \omega > 0$ such that
\begin{align}
    \label{eqn: polynomial-growth-degree-omega}
    c_{1,\omega} \, \|x-x^\star\|_2^\omega
    \le \Phi(x-x^\star)
    \le c_{2,\omega} \, \|x-x^\star\|_2^\omega,
    \qquad \forall x \in \mathbb{R}^d.
\end{align}
This growth assumption has been previously used in various settings in Control and Learning \citep{khalil-book, zaiwei-envelope, hoang-nonlinear-sa, general-purpose-shaan}. Now, observe that if inequality \eqref{eqn: polynomial-growth-degree-omega} of Assumption \ref{assumption: general-drift-condition} holds for some exponent $\omega > 0$, then we have $\Phi_{new} = \Phi^{\frac{2}{\omega}}$ satisfies \eqref{eqn: lyapunov-sandwich} while it admits another negative drift (i.e., \eqref{eqn: negative-drift} of Assumption \ref{assumption: general-drift-condition} for $\eta_{new} =  \frac{2\eta}{\omega}$).

Hence, if we have a Lyapunov function with polynomial growth, we can construct another Lyapunov function with quadratic growth \eqref{eqn: lyapunov-sandwich} in Assumption \ref{assumption: general-drift-condition}.

\textbf{Justification for the smoothness assumption \eqref{eqn: holder-smoothness}:} While not all Lyapunov functions are smooth in practice, \cite{zaiwei-envelope, hoang-nonlinear-sa} has previously shown that one can construct a smooth Lyapunov function satisfying the drift condition \eqref{eqn: negative-drift} and the quadratic growth condition \eqref{eqn: lyapunov-sandwich} using Moreau envelopes. Hence, we proceed to assume our chosen Lyapunov functions are smooth.

Furthermore, we assume the noise to be a martingale difference sequence with zero mean. Motivated by finite-time bound results in \cite{zaiwei-envelope, Chen2025-sa-concentration}, we impose the $p-$th conditional moment of the noise to be either bounded by a constant or have an affine growth with respect to the distance $\|x_k - x^\star\|$. To the best of our knowledge, we are the first to study this setting for the general moment condition and the nonlinear operator, where previous studies either impose additional higher moment information (locally) or a finite constant upper bound that doesn't reflect the usage of SA in applications.

\begin{assumption}[Noise unbiasedness]
\label{assumption: noise-unbiasedness}
Let $\{w_k\}_{k \ge 0}$ be a noise sequence and $\cF_k := \{x_0, w_0, \dots, x_k\}$. Assume that for all $k \ge 0$, we have
\begin{align}
    \mathbb{E}\!\left[w_k \mid \mathcal{F}_k\right] = 0 .
\end{align}
\end{assumption}

\begingroup
\renewcommand{\theassumption}{\arabic{assumption}-p}
\begin{assumption}[$p$-th moment noise growth condition]
\label{assumption: noise-condition-p}
Let $\{w_k\}_{k \ge 0}$ be the noise sequence in
Assumption~\ref{assumption: noise-unbiasedness}.
Assume that there exist constants $A_p > 0$ and $B_p \ge 0$ for some $p > 1$
such that, for all $k \ge 0$,
\begin{align}
\mathbb{E}\!\left[\|w_k\|^p \mid \mathcal{F}_k\right]
\le A_p + B_p \|x_k - x^\star\|^p .
\end{align}
\end{assumption}
\endgroup

When $B_p > 0$, this setting is called \emph{multiplicative noise} and our moment condition generalizes the classical finite-variance assumption ($p = 2$) commonly adopted in the literature \cite{lam-nguyen-hogwild, nonlinear-sa, zaiwei-envelope, sajad-federated-rl, thinh-nonlinear-two-time-sa, hoang-nonlinear-sa, Chen2025-sa-concentration, general-purpose-shaan}. Last but not least, we assume that the operator $H$ is Lipschitz:

\begin{assumption}
    \label{assumption: Lipschitz}
    There exists a constant $C > 0$ such that for any $x,y \in \mathbb{R}^d$:
    \begin{align}
        \|H(x) - H(y)\| \leq C\|x - y\|
    \end{align}
\end{assumption}

To illustrate the necessity of this condition, consider applying stochastic gradient descent to minimize a convex function $F(x)$. By defining the update mapping as $H(x) = x - c \nabla F(x)$, where $c>0$ is a suitably chosen stepsize, the assumption under consideration reduces to the classical Lipschitz continuity of the gradient $\nabla F$. This requirement is well known to play a central role in guaranteeing the stability and convergence properties of the stochastic gradient descent algorithm \cite{lam-nguyen-hogwild,nonlinear-sa, hoang-nonlinear-sa, general-purpose-shaan}.

\subsection{Main results}

We present our main theorem as follows:

\begin{theorem}
    \label{theorem: main-result}
    Suppose Assumptions \ref{assumption: general-drift-condition}, \ref{assumption: noise-unbiasedness}, \ref{assumption: noise-condition-p} for some $p > 1$, and \ref{assumption: Lipschitz} hold. If the step size sequence $\{\alpha_k\}_{k \geq 0}$ in the update \eqref{equation: main_iterate} is non-increasing and satisfies the step size condition \eqref{eqn: step-size-condition}, then the iteration $x_k$ converges almost surely to the unique solution of \eqref{eqn: fixed-point-equation}.
\end{theorem}

Our Theorem \ref{theorem: main-result} generalizes many previous results in the literature while using more relaxed assumptions. In particular, our results hold for multiplicative noise (i.e. $B_p > 0$ in Assumption \ref{assumption: noise-condition-p}), whereas \cite{mertikopoulos2020almost, jin2025stochastic} can only obtain a.s. convergence for additive noise (i.e. $B_p = 0$) or under weak gradient growth assumption. In fact, handling multiplicative noise for $p > 2$ is an extremely non-trivial endeavour, and we will discuss this in depth in our proof outline in Section \ref{sec: proof-outlines} and our detailed proof in Appendix \ref{sec: main-thm-proof}. Additionally, we highlight that our result does not depend on additional moment assumptions, i.e., bounded local $2p-2$-th moment, which is crucial in previous work \cite{mertikopoulos2020almost,jin2025stochastic}. This allows for a wide applicability of our convergence result. Moreover, Theorem \ref{theorem: main-result} can also be viewed as an asymptotic version of the concentration bounds in \cite{Chen2025-sa-concentration}, but for a general $p > 1$.

When the step size $\alpha_k = \alpha(k+K)^{-\xi}$ is chosen, we have the following corollary.

\begin{corollary}
    \label{corollary: valid-xi-range}
    Suppose Assumptions \ref{assumption: general-drift-condition}, \ref{assumption: noise-unbiasedness}, \ref{assumption: noise-condition-p} with $p > 1$ and \ref{assumption: Lipschitz} hold, and the step size has the form $\alpha_n=\alpha(n+K)^{-\xi}$ with $\xi\in\left(\frac{1}{p},1\right]$ and for some $\alpha, K > 0$, we have the iteration $x_k$ converges almost surely to the solution of \eqref{eqn: fixed-point-equation}.
\end{corollary}

\begin{proof}
    The proof follows from the condition \eqref{eqn: step-size-condition}, which implies $\xi \in (p^{-1},1]$. From here, applying Theorem \ref{theorem: main-result} yields the claim.
\end{proof}

In contrast to m.s. convergence where finite-time convergence can be obtained for any polynomially diminishing non-summable step sizes (which corresponds to any $0 < \xi \leq 1$ for $\alpha_k = \alpha(k+K)^{-\xi}$), Corollary \ref{corollary: valid-xi-range} suggests that a.s. convergence is only possible for $\xi \in (p^{-1}, 1]$. In addition, Theorem \ref{theorem: main-result} and Corollary \ref{corollary: valid-xi-range} can be considered as generalizations of the almost sure convergence result in \cite{neurodynamic, zaiwei-envelope, hoang-nonlinear-sa} (for $p = 2$) and a generalization of SLLN (when $p \rightarrow 1^+$) \cite{durrett2019probability}. We refer the readers to Subsection \ref{ssec: example-slln} for a formal discussion of this observation.

Now, if Assumption \ref{assumption: noise-condition-p} holds for any value of $p > 0$, which is the case for sub-Weibullian or bounded a.s. noise sequences, then we can indeed obtain a.s. convergence for any polynomially diminishing non-summable step sizes, as stated in the following corollary.
\begin{corollary}
    \label{corollary: p-infty}
    Suppose Assumptions \ref{assumption: general-drift-condition}, \ref{assumption: noise-unbiasedness}, \ref{assumption: noise-condition-p} for any $p > 1$ and \ref{assumption: Lipschitz} hold, and the step size has the form $\alpha_n=\alpha(n+K)^{-\xi}$ with $\xi\in(0,1]$ and for some $K > 0$. Then the iteration $x_k$ converges almost surely to the solution of \eqref{eqn: fixed-point-equation}.
\end{corollary}

\begin{proof}
    Consider any fixed $\xi \in (0,1)$ and choose any $p > \frac{1}{\xi} > 1$, we have $\xi \in \left(\frac{1}{p}, 1 \right]$ and note that we have Assumption \ref{assumption: noise-condition-p} holds for this choice of $p$. By Theorem~\ref{theorem: main-result}, we have that the iteration converges almost surely. For $\xi = 1$, choosing any $p > 1$ is enough to reach the same conclusion.
\end{proof}

Previous work \cite{caio-meyn2025revisitingstepsizeassumptionsstochastic} establishes almost sure convergence for step sizes of a similar form for every $\xi\in(0,1)$ under parameter-dependent Markovian noise. Their assumptions include the DV3 condition together with additional growth and kernel-regularity conditions that yield moment bounds of every finite order. In contrast, Corollary~\ref{corollary: valid-xi-range} requires only a conditional $p$-th moment bound for a single $p>1$ and guarantees almost sure convergence whenever $\xi\in(1/p,1]$. When these moment bounds hold for every $p>1$, Corollary~\ref{corollary: p-infty} recovers convergence for every $\xi\in(0,1]$.

Now, we will show that having $1 \geq \xi > \frac{1}{p}$ is also a necessary condition to obtain almost sure convergence for polynomially diminishing step sizes. Indeed, we have the following impossibility theorem.

\begin{theorem}
\label{thm:tight}
Consider the step sizes $\alpha_n=\alpha(n+K)^{-\xi}$ with $\xi\in(0,1]$ and for some $K > 0$, there exists a noise process $(w_n)_{n\geq 0}$ satisfying Assumption $\ref{assumption: noise-unbiasedness}$, Assumption \ref{assumption: noise-condition-p} for some $p \geq 1$ and a contraction $H$ such that if $\xi\le 1/p$, the Stochastic Approximation iterates \eqref{equation: main_iterate} fail to converge almost surely.
\end{theorem}

We refer the readers to Appendix \ref{ssec: impossibility-theorem-proof} for the proof of this theorem and Section \ref{sec: applications} for an empirical validation of this result where the iterate indeed can diverge for some choice of noise sequence. Now, to see how Theorem \ref{theorem: main-result} is used, we shall provide concrete demonstrations of our results for contractive operators, linear operators, and gradient operators, which are common settings of Stochastic Approximation \cite{lam-nguyen-hogwild, srikant-ying-td-learning, nonlinear-sa, zaiwei-envelope, zaiwei-constant-stepsize, hoang-nonlinear-sa}, in the following subsections.

\subsection{Example: Contractive operators}
\label{ssec: example-contractive}
Let $w \in \R_{+}^d$, we define $\norm{x}_w = \sqrt{\sum_i w_i x_i^2}$ as a weighted $w$-norm. From here, we call an operator $H: \R^d \rightarrow \R^d$ contractive w.r.t to $w$-norm if there exists $\gamma \in [0,1)$ such that
\begin{align}
    \norm{H(x)-H(y)}_w \leq \gamma \norm{x-y}_w \, \forall x,y \in \R^d.
\end{align}
It is well-known by Banach's Fixed Point Theorem that a contractive operator admits a unique fixed point $x^\star$, and one can obtain the said fixed point $x^\star$ by repeatedly applying the operator \cite{Banach1922}. Contractive operators can be found in many applications in Reinforcement Learning \cite{espeholt2018impala, zaiwei-envelope, zhang-average-nips2021}. To show that our results also apply to contractive operators, we shall show that Assumption \ref{assumption: general-drift-condition} also applies to contractive operators. Indeed, when the operator is contractive with constant $\gamma \in [0,1)$, one can show that the negative drift in Assumption \ref{assumption: general-drift-condition} holds with constant $1-\gamma^2$, thus we can prove the following corollary:
\begin{corollary}
    \label{corollary: contractive-operators}
    Suppose Assumptions \ref{assumption: noise-unbiasedness}, \ref{assumption: noise-condition-p} with $p > 1$ hold, $H$ is a contractive operator w.r.t. to some weighted Euclidean norm $\norm{\cdot}_w$ and choosing the step size $\alpha_k = \alpha (k+K)^{-\xi}$, we have the iteration $x_k$ converges almost surely to the solution of \eqref{eqn: fixed-point-equation} if $\xi \in \left(\frac{1}{p}, 1\right]$.
\end{corollary}

\begin{proof}
    Consider $\Phi(x) = x^{\top}Wx$, we can show that this function satisfies Assumption \ref{assumption: general-drift-condition}. Indeed, Equation \eqref{eqn: holder-smoothness} is proved in \cite{rodomanov2020smoothness}, and let $W = \mathrm{diag}(w_1,\dots,w_d) \succ 0$ and define the product $\langle a, b \rangle_w = a^{\top}Wb$, we have $\norm{x}_w = \sqrt{\langle x, x\rangle_w}$. From here, we have
    \begin{align*}
        &\|H(x) - x^{\star} \|_w^2 \leq \gamma^2 \|x - x^{\star}\|_w^2 \\
        & \|H(x) - x\|_w^2 + 2 \langle H(x) - x, x - x^{\star} \rangle_w \leq (\gamma^2 - 1)\|x - x^{\star}\|_w^2 \\
        \Rightarrow & 2 \langle H(x) - x, x - x^{\star} \rangle_w \leq - (1 - \gamma^2)\|x - x^{\star}\|_w^2. \\
        \Rightarrow & \langle \nabla \Phi(x - x^{\star}), H(x) - x \rangle \leq - (1-\gamma^2)\Phi(x - x^{\star}).
    \end{align*}
    The last inequality is true since:
    \begin{align*}
        2 \langle H(x) - x,\, x - x^{\star} \rangle_w &= \langle \nabla \Phi(x - x^{\star}),\, H(x) - x \rangle.
    \end{align*}
    Finally, we also have Equation \eqref{eqn: lyapunov-sandwich} is satisfied with $c_1 = \lambda_{\min} (W), c_2 = \lambda_{\max} (W)$. Hence, Equation \eqref{eqn: negative-drift} is satisfied with constant $1 - \gamma^2$, it is easy to see that the other conditions are also satisfied naturally by construction. Thus, applying Theorem \ref{theorem: main-result} yields the result.
\end{proof}
On the other hand, in many problems in Reinforcement Learning and Control, the operator $H$ can be a contractive operator with respect to a norm other than Euclidean, example are the $\ell_{\infty}-$norm or weighted Euclidean norm. To this end, suppose that $\|H(x) - H(y)\|_{c} \leq \gamma\|x - y\|_c$ where $\| \cdot\|_c$ is arbitrary norm, we have the following corollary:
\begin{corollary}
    Suppose Assumptions \ref{assumption: noise-unbiasedness}, \ref{assumption: noise-condition-p} with $p > 1$ hold, $H$ is a contractive operator with respect to $\| \cdot \|_c$ and choosing the step size $\alpha_k = \alpha (k+K)^{-\xi}$, we have the iteration $x_k$ converges almost surely to the solution of \eqref{eqn: fixed-point-equation} if $\xi \in \left(\frac{1}{p}, 1\right]$.
\end{corollary}

\begin{proof}
    It is well known in the literature \cite{zaiwei-envelope} that choosing the Moreau envelope $\Phi(x) =M_{f,g}^{\mu}(x) =  \inf_{u \in \R^d} \br{f(u) + \frac{1}{\mu}g(x-u)}$ where $f(x) = \frac{1}{2} \|x\|_c^2, g(x) = \frac{1}{2}\|x\|^2$ and $\mu$ is a sufficiently small positive smoothness parameter satisfies Assumption \ref{assumption: general-drift-condition}, we omit the detail here for clarity. Thus, applying Theorem \ref{theorem: main-result} yields our claim.
\end{proof}

\subsection{Example: Linear operators}
\label{ssec: example-linear}
In this subsection, we will take a look at the case where $H$ is a linear operator, i.e. $H(x) = (A+I)x + b$ for some linear mapping $A$ and $b$ is some vector, $I$ denotes the identity mapping. In this case, our update step is called the Linear Stochastic Approximation, which has the following form.
\begin{align}
    \label{eqn: linear-update-step}
    x_{k+1} = x_k + \alpha_k\br{ Ax_k + b + w_k} \, \forall k \geq 0
\end{align}
where $w_k$ is the noise. The Linear Stochastic Approximation algorithm has found many applications in Reinforcement Learning \cite{srikant-ying-td-learning, zaiwei-envelope, haque2025tightfinitetimebounds}. In order to stabilize the update step, we require the matrix $A$ to be a Hurwitz matrix, that is, a matrix whose eigenvalues all have negative real parts.
\begin{corollary}
    Suppose Assumptions \ref{assumption: noise-unbiasedness}, \ref{assumption: noise-condition-p} with $p > 1$ hold, $A$ is a Hurwitz matrix and choosing the step size $\alpha_k = \alpha (k+K)^{-\xi}$, we have the iteration $x_k$ converges almost surely to the solution of the equation $Ax + b = 0$ if $\xi \in \left(\frac{1}{p}, 1\right]$.
\end{corollary}

\begin{proof}
Since $A$ is Hurwitz, there exist $P,Q$ that are positive definite matrices such that $A^{\top}P + PA = - Q$ \cite{khalil-book, zaiwei-constant-stepsize}. Consider $\Phi(x) = x^{\top}Px$, we can check that:
\begin{align*}
    \bigl\langle \nabla \Phi(x - x^{\star}),\, H(x) - x \bigr\rangle
    &= \bigl\langle (P + P^{\mathsf T})(x - x^{\star}),\, Ax + b \bigr\rangle \\
    &= \bigl\langle (P + P^{\mathsf T})(x - x^{\star}),\, A(x - x^{\star}) \bigr\rangle \\
    &= (x - x^{\star})^{\mathsf T} \br{A^{\top}P + PA} (x - x^{\star}) \\
    &= -\, (x - x^{\star})^{\mathsf T} Q (x - x^{\star}) \\
    &\le -\, \frac{\lambda_{min}(Q)}{\lambda_{max}(P)}\, (x-x^{\star})^{\mathsf T} P (x-x^{\star})
\end{align*}
where $\lambda_{min}, \lambda_{max}$ denote the minimum and maximum eigenvalues of the corresponding matrix. Furthermore, since it's clear that $\Phi(x)$ is smooth (being a quadratic function) and $\lambda_{min}\|x\|^2\leq \Phi(x) \leq \lambda_{max} \|x\|^2$, Assumption \ref{assumption: general-drift-condition} and \ref{assumption: Lipschitz} are satisfied. Applying Theorem \ref{theorem: main-result} yields our claim.
\end{proof}

\subsection{Example: Stochastic Gradient Descent (SGD)}
\label{ssec: example-sgd}
Finally, another special case of the Stochastic Approximation algorithm is the Stochastic Gradient Descent (SGD) algorithm where we wish to find the optimal solution of some objective function $f$. In this case, we are looking for a point $x^\star$ such that $\nabla f(x^\star) = 0$, which is equivalent to
\begin{align}
    H(x^\star) = x^\star - \nabla f(x^\star) = x^\star.
\end{align}
By substituting $H(x) = x-\nabla f(x)$, this gives us the following update step:
\begin{align}
    \label{eqn: sgd-update-step}
    x_{k+1} = x_k + \alpha_k\br{-\nabla f(x_k) + w_k}.
\end{align}
For SGD, a standard analogue of contractivity is the smooth and strong convexity condition, which is formally defined as
\begin{assumption}
    \label{assumption: smooth-strong-convexity-condition}
    A differentiable function $f: \R^d \rightarrow \R$ is said to be $\mu$-strongly convex and $L$-smooth if there exist $L \geq \mu > 0$ such that
    \begin{align}
        f(x) + \langle\nabla f(x),y-x\rangle + \frac{L}{2}\norm{y-x}_2^2 \geq f(y) \geq f(x) + \langle\nabla f(x),y-x\rangle + \frac{\mu}{2}\norm{y-x}_2^2
    \end{align}
    for all $x,y \in \R^d$.
\end{assumption}
Under the gradient smoothness and strong convexity, it is well-known that one can achieve linear convergence for the Gradient Descent algorithm \cite{beck-book}. On the other hand, with the presence of a finite variance noise, one can only obtain $O\br{\frac{1}{\varepsilon}}$ complexity for an $\varepsilon$-approximation problem \citep{zaiwei-envelope, hoang-nonlinear-sa}. For our setting, we can apply Theorem \ref{theorem: main-result} to obtain the following result.
\begin{corollary}
    Under Assumptions \ref{assumption: noise-unbiasedness}, \ref{assumption: noise-condition-p} for $p > 1$ and suppose $f$ satisfies Assumption \ref{assumption: smooth-strong-convexity-condition} with constant $0 < \mu \leq L$. Then by choosing the step size $\alpha_k = \alpha (k+K)^{-\xi}$, we have the iteration $x_k$ converges almost surely to the stationary point $x^\star$ of $f$ if $\xi \in \left(\frac{1}{p}, 1\right]$.
\end{corollary}

\begin{proof}
    Notice that since $f$ is a smooth function, we have $H(x) = -\nabla f(x) + x$ is a Lipschitz operator \cite{ryu2022large}. Moreover, Assumption \ref{assumption: smooth-strong-convexity-condition} satisfies \eqref{eqn: holder-smoothness} and implies the following bound:
    \begin{align}
        \label{eqn: pl-growth-bound}
        \frac{L}{2}\|x - x^{\star}\|^2\geq f(x) - f^{\star} \geq \frac{\mu}{2}\|x - x^{\star}\|^2.
    \end{align}
    Thus, choosing $\Phi(x - x^{\star}) = f(x) - f^{\star}$, we have $\Phi$ satisfies Assumption \ref{assumption: general-drift-condition}, since the negative drift inequality \eqref{eqn: negative-drift} follows from the $\mu$-strongly convex condition as:
    \begin{align}
        \langle \nabla \Phi(x- x^{\star}),H(x) - x \rangle = - \langle  \nabla f(x), \nabla f(x) \rangle \leq -2\mu (f(x) - f^{\star}) = -2\mu \Phi(x - x^{\star})
    \end{align}
    and the growth condition \eqref{eqn: lyapunov-sandwich} follows from \eqref{eqn: pl-growth-bound}. Thus, applying Theorem \ref{theorem: main-result} yields our claim.
\end{proof}
For most of the SGD convergence guarantees in the literature, they are typically established for square-integrable noises, which is equivalent to the case $p = 2$. As it turns out, many applications in ML exhibit heavy-tailed behaviors \cite{infinite-variance-sgd-mert, tail-index-simsekli19a}, i.e., infinite variance noise. Previously, it was shown that by considering a simple one-dimensional quadratic optimization problem with heavy-tailed noise, SGD iterations can diverge in expectation. Thus, this leads researchers to design variants of SGD that can provably converge in expectation using additional mechanisms such as gradient clipping or normalization. A recent study \cite{fatkhullin2025can} takes a deeper look where the authors show that in the bounded-domain setting, a vanilla average of SGD with proper step-size tuning can converge in expectation with the optimal rate, albeit they can only handle additive noise. Nonetheless, they also show by construction that any output of SGD without additional modification results in divergence in expectation in the unbounded-domain setting. Our work offers an alternative perspective on the convergence of the Stochastic Approximation algorithm by showing that even for unbounded domain settings, multiplicative noise, and non-expansive operators, the Stochastic Approximation iteration and its special case SGD can still converge almost surely to the solution given an appropriate choice of step size.

\subsection{Example: Strong Law of Large Numbers (SLLN) with step-sizes}
\label{ssec: example-slln}
The SLLN is a classic result in Probability theory, which states that the long-run average of Independent and Identically distributed random samples converges to its expected value when it exists. More formally, let $Z_0, Z_1, \dots$ be an infinite sequence of independent and identically distributed (i.i.d.) random variables with a finite expected value $E[Z_i] = \mu$. The sample average is defined as:
\begin{align}
    \label{eqn: sample-average}
    X_n = \frac{1}{n} \sum_{i=0}^{n-1} Z_i
\end{align}
SLLN states that the sample average converges almost surely to $\mu$, that is $X_n \overset{a.s.}{\rightarrow} \mu$. Now, note that \eqref{eqn: sample-average} can be written in an iterative manner starting from $X_{1} = Z_0$ and for $n\geq 1$ as follows
\begin{align}
    X_{n+1} = X_n + \frac{1}{n+1} \br{Z_{n}-X_n},
\end{align}
which can be viewed as a special form of the Stochastic Approximation update \eqref{equation: main_iterate} for $H \equiv \mu$, $w_n=Z_n-\mu$ as the unbiased noise term, so that $H(x_n)+w_n=Z_{n}$, and step size $\alpha_n = \frac{1}{n+1}$.  Here, the limit point $\mu$ can be viewed as the unique fixed point of $H$ and the step size corresponds to the case $\xi = 1$. A natural question is whether one can generalize SLLN to the setting of other choices of step sizes. More precisely, the question is if  the recursion defined by
\begin{align}
    \label{eqn: iterative-slln}
    X_{n+1} = X_n + \alpha_n \br{Z_{n}-X_n},
\end{align}
also converges almost surely to the mean $\mu$ for other choices of step-sizes $\alpha_n$ beyond the case of classical SLLN where $\alpha_n=\frac{1}{n+1}$. Indeed, this immediately follows from our theorem, which we now state as a corollary.

\begin{corollary}
    \label{corollary: weighted-slln}
    Let $Z_0, Z_1, \dots$ be an infinite sequence of i.i.d. scalar random variables with a finite expected value $E[Z_0] = \mu$ and $\E[|Z_0|^p] < \infty$ for some $p > 1$. Then, the recursion \eqref{eqn: iterative-slln} with initialization $X_{1} = Z_0$ converges almost surely to $\mu$ as long as the step-size condition \eqref{eqn: step-size-condition} is met. In particular, this holds for the step sizes $\alpha_k = \alpha(k+K)^{-\xi}$ for some $\alpha, K > 0$ and $\xi \in (p^{-1},1]$.

\end{corollary}

\begin{proof}
    Clearly, the recursion \eqref{eqn: iterative-slln} is a special case of SA \eqref{equation: main_iterate} where $H \equiv \mu$, $w_n=Z_n-\mu$ as the unbiased noise term, so that $H(x_n)+w_n=Z_{n}$. Note that $H$ is a constant operator, and so, it is a contractive operator with parameter $\gamma = 0$.

    Moreover, the noise term $w_n=Z_n-\mu$
    is indeed unbiased and has a finite $p$-th moment since by Holder's inequality, we have
    \begin{align}
        \E\sqbr{|Z_n-\mu|^p} \leq \E\sqbr{(|Z_n|+|\mu|)^p} \leq 2^{p-1} (\E[|Z_n|^p]+|\mu|^p) < \infty.
    \end{align}
    Thus, $X_n$ converges almost surely to $\mu$ by Corollary \ref{corollary: contractive-operators}.
\end{proof}
Indeed, this corollary for the case of $p=2$ can be found in Theorem 14.5 of \cite{mohri2018foundationsmachinelearning-book}. Thus, we generalize it to the setting of $p>1$.

Note that unfortunately, we cannot obtain the classic SLLN as a special case when $p=1$. This is because it corresponds to the case of $\xi=1/p$. This case, sits at a knife-edge, and the convergence behavior of SA when $\xi=1/p$ is complex. In the proof of Theorem \ref{thm:tight}, we present a counterexample where the SA update \eqref{equation: main_iterate} does not converge almost surely when the noise sequence is a particular martingale difference sequence. However, in SLLN, the noise is i.i.d., and so the counterexample does not preclude convergence in the i.i.d. noise case. Thus, characterizing conditions for convergence in the case of $\xi=1/p$ is an interesting future direction.

A different generalization of SLLN in terms of weights was investigated in \cite{fazekas2017note}, and it is not directly related to the SLLN with steps-sizes that we present here.

\subsection{Generalization to non-expansive operators}
\label{ssec: non-expansive-operators}
While Assumption \ref{assumption: general-drift-condition} is covered in many applications, some important classes of problems, such as convex optimization with SGD \cite{boyd-primer-operator} or average-reward Q-learning \cite{abounadi2002stochastic-non-expansive,he2022emphatic}, generally do not admit such negative drift. Thus, beyond this assumption, we can further generalize our results to operators with milder conditions, such as non-expansive operators. We call an operator $H: \R^d \rightarrow \R^d$ non-expansive with respect to Euclidean norm $\|\cdot \|_2$ if
\begin{align}
    \norm{H(x)-H(y)}_2 \leq \norm{x-y}_2 \, \forall x,y \in \R^d.
\end{align}
In the SGD special case, one can show that if the objective function $f$ is convex and smooth with parameter $L$, then the operator $H(x) = x-\eta\nabla f(x)$ is indeed non-expansive \cite{boyd-primer-operator} for $\eta \in (0, \frac{2}{L})$. For non-expansive operators, we do not necessarily have a unique fixed point $x^\star$ (if any even exists). Hence, we have the following assumption.
\begin{assumption}[Bounded solution]
    \label{assumption: bounded-solution}
    The set $\cX = \{ x \in \R^d \mid H(x) = x \}$ is bounded with diameter $\cD$ and nonempty.
\end{assumption}
Since we do not have uniqueness of $x^\star$, the approach of taking $x_{n+1} = H(x_n)$ as in Banach's Fixed Point Theorem no longer works (for instance, consider $H$ as the reflection operator w.r.t. to some plane). Nevertheless, by introducing diminishing step sizes, one can show that the iterates can still converge to the fixed point set of $H$ \cite{boyd-primer-operator,zaiwei-envelope}. Previously, almost sure convergence guarantees for non-expansive operators for $p = 2$ were established in \cite{zaiwei-envelope, bravo-cominetti-non-expansive}. Here, we obtain the following extension for all $p > 1$ for non-expansive operators.

\begin{theorem}
    \label{theorem: nonexpansive-a.s.-convergence}
    Let $p>1$. Suppose that Assumptions \ref{assumption: noise-unbiasedness}, \ref{assumption: noise-condition-p} with $p > 1$ hold for some $B_p \geq 0$ when $p\in(1,2]$ and $B_p=0$ when $p>2$. Moreover, suppose Assumption \ref{assumption: bounded-solution} holds, and that $H$ is non-expansive with respect to the Euclidean norm. If the positive step-size sequence $\{\alpha_k\}_{k\geq 0}$ is non-increasing and satisfies the step size condition \eqref{eqn: step-size-condition} then $\lim_{k \to \infty} \operatorname{dist}(x_k, \cX)  = 0$ almost surely.

\end{theorem}
We defer the proof of Theorem \ref{theorem: nonexpansive-a.s.-convergence} to Appendix \ref{sec: non-expansive-proof-heavy-tailed-regime}. Theorem \ref{theorem: nonexpansive-a.s.-convergence} shows that we can also obtain almost sure convergence beyond Assumption \ref{assumption: general-drift-condition}. Notably, we can obtain almost sure convergence for multiplicative noise sequences when $p \in (1,2]$, whereas we only have a guarantee for additive noise sequences when $p \in (2,\infty)$. This, in fact, is not an artifact of our proof. The following proposition gives a counterexample with a non-expansive operator that has a unique fixed point and with noise satisfying the multiplicative $p$-th moment condition such that $\eqref{equation: main_iterate}$ fails to converge almost surely as follows.

\begin{proposition}
\label{proposition: non-expansive-multiplicative-counterexample}
Fix $p>2$.  There exist a non-expansive operator $H:\mathbb{R}^{2}\to\mathbb{R}^{2}$, a martingale-difference noise sequence $\{w_k\}_{k\geq 0}$ satisfying Assumptions \ref{assumption: noise-unbiasedness}, \ref{assumption: noise-condition-p}, and a non-increasing step-size sequence $\{\alpha_k\}_{k\geq 0}$ satisfying
\begin{align}
    \sum_{k=0}^{\infty}\alpha_k=\infty, \qquad \sum_{k=0}^{\infty}\alpha_k^p<\infty,
\end{align}

but the Stochastic Approximation iterates \eqref{equation: main_iterate} fail to converge almost surely.
\end{proposition}

When $p > 2$, it is now possible for the step size sequence $\{\alpha_k\}$ to be no longer square-summable, and so the $\ell_2$ distance of the iterates to the solution set can diverge with an appropriate choice of noise and operator $H$. Nevertheless, this counterexample provides insights into the difficulties of analyzing non-expansive operators, as most previous convergence results on non-expansive operators only cover additive noise sequences \citep{zaiwei-envelope, bravo-cominetti-non-expansive, liu2025extensions, blaser2026asymptotic-non-expansive}.

\section{Proof outline}
\label{sec: proof-outlines}

Previously, the finite-time behavior of the discrete-time Stochastic Approximation method was commonly analyzed using the celebrated Lyapunov drift method \cite{srikant-ying-td-learning, nonlinear-sa, zaiwei-envelope, sajad-two-time-scale-tac, thinh-nonlinear-two-time-sa, hoang-nonlinear-sa}. In particular, these prior works aim to construct a suitable potential function (also known as a Lyapunov function) so that they can obtain a tight one-step drift, from which they can obtain tight finite-time bounds under an appropriate choice of the step sizes. The strength of the Lyapunov drift method lies in its versatility and its ability to handle nonlinear operators without needing to expand the iterates. For our case, the so-called Almost Supermartingale Convergence Theorem \cite{robbins1971convergence,neurodynamic} is the bread-and-butter of the Lyapunov drift argument to establish almost sure convergence. We shall state the theorem formally as follows.

\begin{theorem}[Almost supermartingale convergence]
    \label{theorem: sct_theorem}
    Let $\{\mathcal{F}_t\}_{t \ge 0}$ be a filtration of sigma-algebras on a probability space, so that $\mathcal{F}_t \subseteq \mathcal{F}_{t+1}$ for every $t \ge 0$. Let $\{Y_t\}_{t \ge 0}$, $\{X_t\}_{t \ge 0}$, $\{Z_t\}_{t \ge 0}$, and $\{\beta_t\}_{t \ge 0}$ be sequences of random variables. Suppose that:
\begin{enumerate}[label=(\alph*)]
    \item For every $t \ge 0$, the random variables $Y_t, X_t, Z_t, \beta_t$ are $\mathcal{F}_t$-measurable, nonnegative, and finite almost surely.
    \item For every $t \ge 0$, the following inequality holds almost surely:
    \begin{align*}
        \mathbb{E}[Y_{t+1} \mid \mathcal{F}_t] \le (1+\beta_t)Y_t - X_t + Z_t.
    \end{align*}
    \item $\sum_{t=0}^{\infty} Z_t < \infty$ almost surely.
    \item $\sum_{t=0}^{\infty} \beta_t < \infty$ almost surely.
\end{enumerate}
Then we have that $\{X_t\}_{t \ge 0}$ is summable and $\{Y_t\}_{t \ge 0}$ converges to a nonnegative random variable $Y$ with probability~1.
\end{theorem}

Essentially, Theorem \ref{theorem: sct_theorem} says that it is sufficient to establish a "nice enough" one-step drift relation to obtain almost sure convergence. To do so, one needs to choose a suitable non-negative potential function $V$ such that $Y_t = V(x_t)$ gives rise to the appropriate one-step drift that satisfies all conditions of Theorem \ref{theorem: sct_theorem}.

Naturally, the chosen function should contain information related to the quantity $\|x_k - x^\star\|$ so that when proving $V(x) \rightarrow 0$, our desired result follows.

To this end, suppose that Assumption \ref{assumption: general-drift-condition} holds for some potential function $\Phi$ and Assumption \ref{assumption: noise-condition-p} holds with $p > 1$, we shall explain the outline of our proof below.

\subsection{\texorpdfstring{$p = 2$}{p = 2}}
While this is a known result \cite{robbins1971convergence,Borkar2008StochasticAA,zaiwei-envelope,hoang-nonlinear-sa}, we shall do a brief recap for the case $p = 2$ as a starting point. In this case, using $\Phi$ in Assumption \ref{assumption: general-drift-condition} directly as the Lyapunov function and applying the smoothness inequality \eqref{eqn: holder-smoothness}, we can establish a ``drift inequality'' as follows:
\begin{align}
    \label{equation: quadratic-drift}
    \E[\Phi(x_{k+1} - x^\star) \mid \cF_{k}] \leq \Phi(x_k - x^\star) - \tilde{c}_1 \alpha_k \Phi(x_k - x^\star) + \tilde{c}_2 \alpha_k^2 \E[\|w_k\|^2 \mid \cF_k].
\end{align}
Which can be further bounded by Assumption \ref{assumption: noise-condition-p} as
\begin{align}
    \label{equation: quadratic-drift-2}
    \E[\underbrace{\Phi(x_{k+1} - x^\star)}_{Y_{k+1}} \mid F_{k}] \leq (1+\underbrace{\tilde{c}_2 B_p \alpha_k^2}_{\beta_k})\underbrace{\Phi(x_k - x^\star)}_{Y_k} - \underbrace{\tilde{c}_1 \alpha_k \Phi(x_k - x^\star)}_{X_k} + \underbrace{\tilde{c}_2 A_2 \alpha_k^2}_{Z_k}.
\end{align}
By choosing the step size $\alpha_k$ such that $\sum \alpha_k^2 < \infty$ and $\sum \alpha_k = \infty$, we can show that both $\{\beta_k\}_{k \geq 0}$ and $\{Z_k\}_{k \geq 0}$ are summable and thereby obtain the a.s. convergence result.

\subsection{\texorpdfstring{$p \in (1,2)$}{1 < p < 2}}
For $p = 2$, the proof follows naturally since we can apply the smoothness inequality directly to establish \eqref{equation: quadratic-drift}, which is similar to how one would treat $\norm{x-x^\star}^2$ if it were the Lyapunov function.

For $p \in (1,2)$ or $p \neq 2$ generally, one would like to choose a potential function $V$ such that it behaves similarly to $\norm{x-x^\star}^p$ so that it can give rise to a one-step drift in the form of
\begin{align}
    \E\sqbr{V(x_{k+1}) | \cF_k} \leq (1+\tilde{c}_1 \alpha_k^p) V(x_k) - \tilde{c}_2 \alpha_k V(x_k) + \tilde{c}_3 \alpha_k^p.
\end{align}
Luckily, by choosing $V=\Phi^{\frac{p}{2}}$, we can obtain a negative drift similar to \eqref{eqn: negative-drift} for $V$ and $V(x) = \Theta\br{\norm{x-x^\star}^p}$ from \eqref{eqn: lyapunov-sandwich}. Moreover, we can actually prove a Holder smoothness condition for $V$ in the following form
\begin{align}
    \label{eqn: p-holder-smoothness-proof-outline}
    V(u)\le V(v)+\langle\nabla V(v),u-v\rangle+\frac{L_p}{p}\|u-v\|^p \, \forall u,v \in \R^d.
\end{align}
By choosing $v = x_{k} - x^\star$ and $u = x_{k+1} - x^\star$, and taking conditional expectations on both sides, we can establish such a drift inequality. In general, by proving that $\Phi^{\frac{p}{2}}$ has negative drift with constant $\eta_p$, we have for a sufficiently large $k$:
\begin{align}
    \E\sqbr{\Phi(x_{k+1} - x^{\star})^{\frac{p}{2}} \big| \cF_k} \leq \Phi(x_k - x^{\star})^{\frac{p}{2}} - \tilde{c}_3 \alpha_k \Phi(x_k - x^{\star})^{\frac{p}{2}} + c_4\alpha_k^p \E\sqbr{\|w_k\|^p \big| \cF_k}.
\end{align}

From Assumption \ref{assumption: noise-condition-p}, this becomes
\begin{align}
\label{equation:p-drift}
    \E\!\left[\underbrace{\Phi(x_{k+1} - x^{\star})^{\frac{p}{2}}}_{Y_{k+1}} \,\middle|\, \cF_k\right] \le \br{1 + \underbrace{c_4 B_p \alpha_k^p}_{\beta_k}} \underbrace{\Phi(x_k - x^{\star})^{\frac{p}{2}}}_{Y_k} - \underbrace{\tilde{c}_3 \alpha_k \Phi(x_k - x^{\star})^{\frac{p}{2}}}_{X_k} + \underbrace{c_4 A_p \alpha_k^p}_{Z_k}.
\end{align}
Since we have $\sum \alpha_k = \infty$ and $\sum \alpha_k^p < \infty$, applying Theorem~\ref{theorem: sct_theorem} yields our claim.

\subsection{\texorpdfstring{$p > 2$}{p > 2}}
While the proof strategy for $p \in (1,2]$ is rather straightforward, such a strategy is no longer applicable for $p > 2$.

As previously explained in the case $p \in (1,2)$, we would like to choose the potential function $V$ such that it behaves similarly to $\norm{x-x^\star}^p$. However, this gives rise to our second challenge, that is, such a Lyapunov function would produce a lot of nonlinear cross-terms when applying a Taylor-like bound for $p > 2$. To illustrate our point, we shall give concrete calculations for the simple case $p = 4$, as it is natural to attempt extending the bound for $p = 2$ to the case $p = 4$.

\subsubsection{Challenge}
We present the challenge that arises when using $V(x) = \norm{x-x^\star}^p$ directly by considering the case $p = 4$ as the motivating example where we expect step sizes satisfying $\sum \alpha_k = \infty$ and $\sum \alpha_k^4 < \infty$ are sufficient to obtain almost sure convergence. However, unlike the proof in the $p \in (1,2]$ case, using this Lyapunov function will in fact, yield a suboptimal guarantee even for the simple setting where $H \equiv 0$, and the iterate $x_k$ is one-dimensional. Indeed, we have the following proposition.
\begin{proposition}[Fourth--order Lyapunov drift, tightness]
\label{prop:fourth-drift}
Consider $(x_k)_{k\ge 0}$ satisfy
\begin{align*}
x_{k+1} = x_k + \alpha_k(-x_k + w_k),
\end{align*}
and let $\cF_k := \{x_0,w_0,\dots,x_k\}$.
Assume
\begin{align*}
    \E[w_k \mid \cF_k] = 0, \qquad
    \E[w_k^2 \mid \cF_k] \le A_2, \qquad
    \E[|w_k|^3 \mid \cF_k] \le A_3, \qquad
    \E[w_k^4 \mid \cF_k] \le A_4.
\end{align*}
Then there exist constants $c,C'>0$ such that, for all
sufficiently large $k$, we have
\begin{align*}
    \E[x_{k+1}^4 \mid \cF_k] \le x_k^4 - c\,\alpha_k x_k^4 + C'\,\alpha_k^3.
\end{align*}
Consequently, if one has $\sum_{k=1}^\infty \alpha_k = \infty$ and $\sum_{k=1}^\infty \alpha_k^3 < \infty$, the iteration $x_k$ converges almost surely.
\end{proposition}
We defer the complete calculations of Proposition \ref{prop:fourth-drift} to Appendix \ref{sec: fourth-drift-proof}. As evident by the proposition and its proof, the strategy of directly using the Lyapunov function $\|x_k - x^\star\|^p$ can only yield suboptimal guarantees, and so our proof strategy here fails. In particular, if we choose the step size $\alpha_k = \alpha(k+K)^{-\xi}$, we can only guarantee $\xi > \frac{1}{3}$ yields almost sure convergence with this guarantee. This is due to the fact that when $p > 2$, applying Taylor-like bounds to the aforementioned Lyapunov function would produce many nonlinear cross-terms which are extremely non-trivial to handle. Recognizing this challenge, we propose the following proof strategy.

\subsubsection{Our proof strategy}

We begin by making some observations on the iterates. If $\norm{x_k - x^\star}$ is sufficiently large, or equivalently, $\Phi(x_k - x^\star)$ is sufficiently large, then we can obtain a contraction for small enough $\alpha_k$ (i.e., large enough $k$ for diminishing step sizes)
\begin{equation}
    \E[\Phi(x_{k+1} - x^\star) \mid \mathcal F_k] \leq (1 - c_1'\alpha_k)\Phi(x_k - x^\star).
\end{equation}
Hence, a simple bound on the quadratic function would suffice here. On the other hand, when $\norm{x_k-x^\star}$ is "small", the effect from the noise is comparatively greater, and hence it is harder to control the behavior of $u_k = \sqrt{\Phi(x_k - x^\star)}$ as the iterates exhibit oscillating behavior in this regime.

From these observations and to resolve the challenge of handling nonlinear cross-terms highlighted in Proposition \ref{prop:fourth-drift}, we introduce a Lyapunov function of the form
\begin{align}
    z_k = \max(\min(2D,u_k),D) - D.
\end{align}
Effectively, $z_k + D$ is exactly the projection of $u_k$ onto the interval $[D,2D]$. The key idea here is that if $x_k$ converges to $x^\star$ almost surely (i.e., $u_k$ converges to $0$ almost surely), then $u_k$ can only cross the interval $[D,2D]$ finitely many times before staying in $[0,D)$ indefinitely. Thus, we reduce the problem to first showing the almost sure convergence of $z_k$. Once this is established, all that remains is to show that if $z_k$ converges almost surely, then it has to converge almost surely to $0$, which follows a standard argument as in \citep{neurodynamic, zaiwei-envelope}.

To show the almost sure convergence of $z_k$, it is sufficient to establish a one-step bound in the form of
\begin{align}
    \E\sqbr{z_{k+1}^p | \cF_k} \leq z_k^p + O(\alpha_k^p).
\end{align}
To this end, we will break into $3$ phases as follows:

\textbf{Contraction phase ($u_k \geq 2D$):} In this case, due to a strong drift, we have $\E[z_{k+1}^p | \cF_k] \leq z_k^p$.

\textbf{Fluctuation phase ($D \leq u_k < 2D$):} Here, we have $\min(2D,u_{k+1})-u_k \leq u_{k+1}-u_k$, and so we show the negative drift $u_{k+1}-u_k$ is sufficiently large to make some progress, but not enough to cancel out the noise effect. In this case, one can show that $\E\sqbr{z_{k+1}^p | \cF_k} \leq z_k^p + O(\alpha_k^p)$.

\textbf{Stabilization phase ($0 \leq u_k < D$):} In this case, we have the effect of the noise dominates, and we can show that $\E[z_{k+1}^p | \cF_k] \leq O(\alpha_k^p)$ and as $\alpha_k \rightarrow 0$, the iterates $z_k$ will vanish to $0$.

From here, we can let $D \rightarrow 0$ to show that $u_k \rightarrow 0$, which implies $x_k \rightarrow x^\star$ almost surely.
\vspace{-5pt}

\begin{figure}[htbp]
\centering
\begin{tikzpicture}[scale=1.1]

\draw[->] (-.5,0) -- (7.5,0) node[right] {$k$ (iteration)};
\draw[->] (0,-.3) -- (0,3.5) node[above] {$\|x_k-x^\star\|$};

\fill[gray!15] (0,1) rectangle (7.5,2);
\draw[dashed] (0,1) node[left] {$D$} -- (7.5,1);
\draw[dashed] (0,2) node[left] {$2D$} -- (7.5,2);

\draw[dashed] (1.9,2) -- (1.9,3.5);
\draw[dashed] (5.75,1) -- (5.75,3.5);
\node[black,right] at (0.75,3.3) {\small \eqref{case: contraction}};
\node[black,right] at (0.2,3) {\small Contraction};
\node[black,right] at (3.5,3.3) {\small \eqref{case: fluctuation}};
\node[black,right] at (3,3) {\small Fluctuation};
\node[black,right] at (6.5,3.3) {\small \eqref{case: stabilization}};
\node[black,right] at (5.8,3) {\small Stabilization};

\draw[very thick,blue,->]
plot[
  domain=0:7,
  samples=300,
  variable=\x
]
(
  {\x},
  {2.3
   + 0.7*exp(-0.15*\x)*sin(80*\x)
   - 0.28*\x}
);
\node[blue,above right] at (7,0.25) {\small $u_k$};

\draw[very thick,red,-]
plot[
  domain=0:1.9,
  samples=300,
  variable=\x
]
(
  {\x},
  {2}
);

\draw[very thick,red,-]
plot[
  domain=1.9:3.15,
  samples=300,
  variable=\x
]
(
  {\x},
  {2.3
   + 0.7*exp(-0.15*\x)*sin(80*\x)
   - 0.28*\x}
);

\draw[very thick,red,-]
plot[
  domain=3.15:4.3,
  samples=300,
  variable=\x
]
(
  {\x},
  {1}
);

\draw[very thick,red,-]
plot[
  domain=4.3:5.75,
  samples=300,
  variable=\x
]
(
  {\x},
  {2.3
   + 0.7*exp(-0.15*\x)*sin(80*\x)
   - 0.28*\x}
);

\draw[very thick,red,->]
plot[
  domain=5.75:7,
  samples=300,
  variable=\x
]
(
  {\x},
  {1}
);

\node[red,right] at (3.2,1.5) {$z_k+D$};
\end{tikzpicture}
\caption{Illustration of the iterates $u_k$ and its projected iterates $z_k+D$ in the corresponding phases.}
\label{fig:trap}
\end{figure}

\section{Experiments}
\label{sec: applications}

In this section, we consider the potential application of our result for non-smooth exponentially stable systems, particularly those that satisfy Assumption \ref{assumption: general-drift-condition} with a piecewise Lyapunov function. Furthermore, this system is chosen precisely because there is no global quadratic Lyapunov function that admits the negative drift condition. More precisely, we follow the setting in \cite{Johansson2003PiecewiseLC,hoang-nonlinear-sa} and consider the recursion
\begin{align*}
    x_{k+1}=x_k+\alpha_k\br{Ax_k+B\max\{k^\top x_k,0\}+w_k},
\end{align*}
where $\{\alpha_k\}_{k \geq 0}$ is a positive step--size sequence and $\{w_k\}_{k \geq 0}$ is a martingale difference noise. Let
\begin{align*}
    A=\begin{bmatrix}-5&-4\\-1&-2\end{bmatrix},\quad
    B=\begin{bmatrix}3\\21\end{bmatrix},\quad
    k=\begin{bmatrix}1\\0\end{bmatrix}.
\end{align*}
Then the drift is piecewise linear with $f(x)=A_1x$ for $k^\top x\le0$ and $f(x)=A_2x$ otherwise, where $A_1 = A, A_2 = A + Bk^T$. Define the Lyapunov function
\begin{align*}
V(x)=
\begin{cases}
    x^\top Px, & k^\top x\le0,\\
    x^\top(P+\eta kk^\top)x, & k^\top x>0,
\end{cases}
\end{align*}
with
\begin{align*}
    P=\begin{bmatrix}1&0\\0&3\end{bmatrix},\qquad
    \eta=9,\qquad
    P+\eta kk^\top=\begin{bmatrix}10&0\\0&3\end{bmatrix}.
\end{align*}
Both $P$ and $P+\eta kk^\top$ are positive definite and satisfy
\begin{align*}
    A_1^\top P+PA_1<0,\qquad
    A_2^\top(P+\eta kk^\top)+(P+\eta kk^\top)A_2<0,
\end{align*}
From this, the author in \cite{hoang-nonlinear-sa} constructs a smooth Moreau envelope approximation $\Phi(x)$ to $V(x)$ that satisfies Assumption \ref{assumption: general-drift-condition}, which we omit here for clarity. To this end, let $\cF_k = \{x_0, w_0, ..., x_k\}$, define the noise sequence by
\begin{align*}
    w_k=s_k \zeta_k \begin{bmatrix}1\\1\end{bmatrix},
    \qquad \P(\zeta_k=\pm1)=q_k,
    \qquad \P(\zeta_k=0)=1-2q_k,
\end{align*}
where
\begin{align*}
    s_k=\frac{4}{\alpha}(k+K)^{\xi}, \qquad
    q_k=c(k+K)^{-\xi p}, \qquad c\in(0,1/2].
\end{align*}

Then, we have $\E[w_k \mid \cF_k]=0$ and $\E\!\left[\|w_k\|^p\mid\cF_k\right] = 2^{1+p/2} q_k s_k^p = 2^{1+p/2} c\left(\frac{4}{\alpha}\right)^p<\infty$, which satisfy Assumptions \ref{assumption: noise-unbiasedness}, \ref{assumption: noise-condition-p}. Letting $p = 1.6$, we follow this exact construction to construct the noise for our experiment, testing convergence for $\xi \in (p^{-1},1]$. The results are shown for the three plots on the top row of \ref{fig: plot_2}.

\begin{figure}[htbp]
    \centering
    \includegraphics[width=0.95\textwidth]{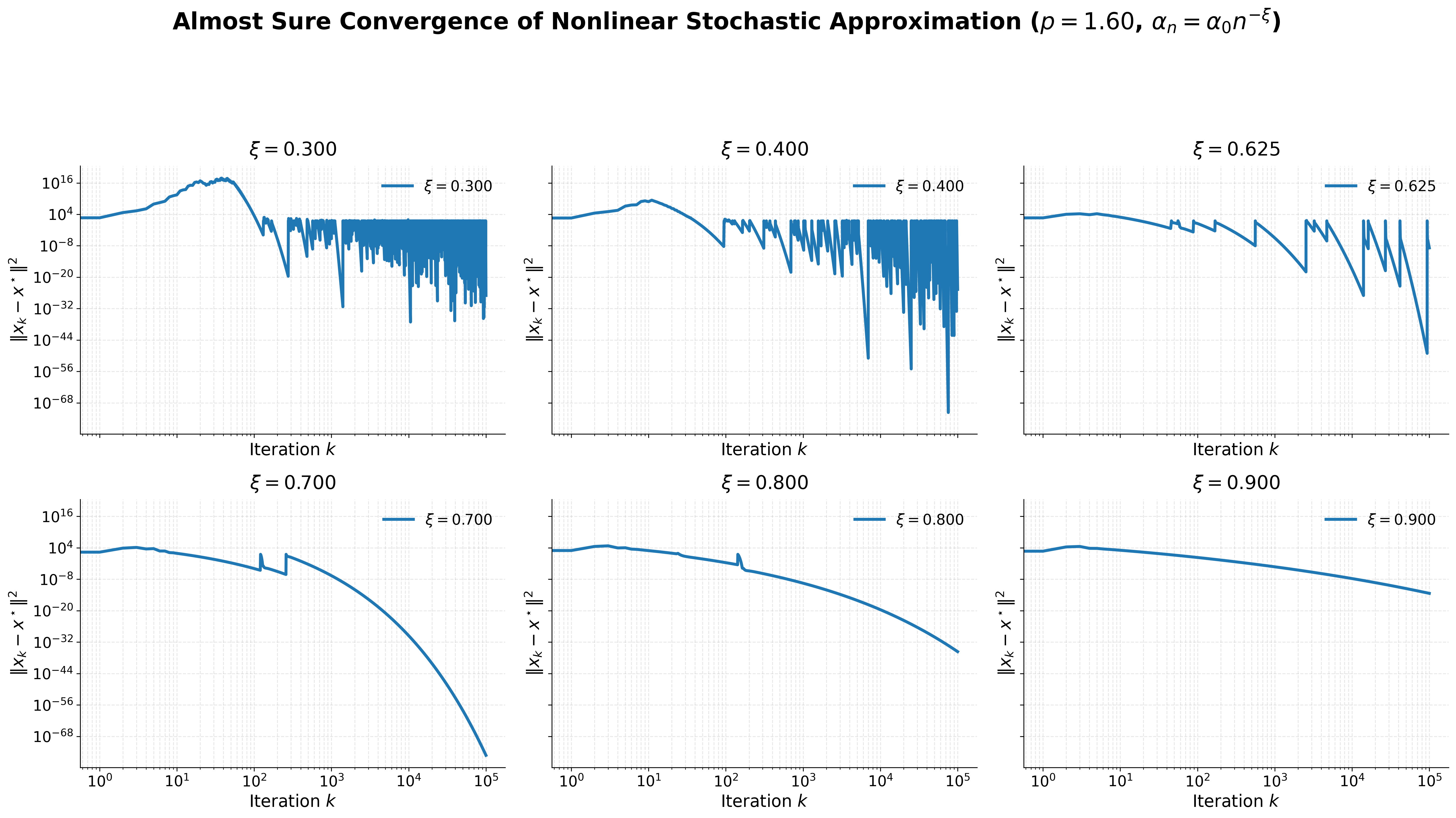}
    \caption{Selector control with theoretical noise and multiple step sizes, here $p^{-1} = 0.625$ is the divergence threshold for $p = 1.6$. Our experimental results confirm the tightness of the condition $\xi > p^{-1}$ for $p = 1.6$.}
    \label{fig: plot_2}
\end{figure}

To empirically verify Theorem \ref{thm:tight}, we reuse the same noise construction, with the only change being the choice of $\xi \in (0,\frac{1}{p}]$. The results are shown in Figure ~\ref{fig: plot_2}. It can be observed that the experiments clearly capture the convergence behavior associated with each choice of step size. For $\xi \in (0,\tfrac{1}{p}]$, the algorithm exhibits infinitely many large fluctuations, leading to divergence, whereas this behavior does not occur for $\xi \in (\tfrac{1}{p},1]$. Moreover, the increased robustness obtained by choosing smaller step sizes, as predicted by the finite-time analysis, is also evident. Overall, the experimental results are consistent with the theoretical predictions.

\section{Conclusion and Future Work}
While the Strong Law of Large Numbers and the analysis of almost sure convergence of Stochastic Approximation are very well-studied topics in Applied Probability, establishing these results for a general noise condition is highly non-trivial. In our work, we generalize prior results on almost sure convergence of Stochastic Approximation \cite{neurodynamic, Borkar2008StochasticAA, zaiwei-envelope, mertikopoulos2020almost, jin2025stochastic} by establishing almost sure convergence of the Stochastic Approximation algorithm with nonlinear operators under a general drift condition and a general noise condition. We establish such results by utilizing the Lyapunov drift framework, which allows us to obtain almost sure convergence without expanding the iterates. To handle complications in the case $p > 2$, we introduce a novel iterate projection technique that significantly simplifies the drift analysis. We believe that our results and techniques could pave the way for many exciting future works. In particular, one immediate extension of our work is to obtain finite-time m.s. guarantees for the Stochastic Approximation algorithm under a general noise condition. From there, one can attempt to generalize the concentration results in \cite{Chen2025-sa-concentration, khodadadian2025a-concentration} or apply the analysis to different noise profiles such as Markovian noise.

\bibliographystyle{plainnat}
\bibliography{refs}

\appendix

\section{Proof of Theorem \ref{theorem: main-result}}
\label{sec: main-thm-proof}
\begin{proof}
Because of the noise condition, we have that for each $0 < \kappa \leq p$, there exist constants $A_{\kappa}, B_{\kappa}$ such that $\E[\|w_k\|^{\kappa} | \cF_k] \leq A_{\kappa} + B_{\kappa}\|x_k - x^\star\|^{\kappa}$. Indeed, for any $\kappa \leq p$, by conditional Jensen,
\begin{align}
    \E\!\left[\|w_k\|^{\kappa}\mid \cF_k\right] \le \left(\E\!\left[\|w_k\|^{p}\mid \ \cF_k\right]\right)^{\kappa/p} \le (A_p + B_p\|x_k-x^\star\|^{p})^{\kappa/p} \le A_{\kappa}+B_{\kappa}\|x_k-x^\star\|^{\kappa}
\end{align}
where the last inequality used the fact that $(x+y)^{\kappa/p} \leq x^{\kappa/p} + y^{\kappa/p}$ for $x,y > 0$ since $\frac{\kappa}{p} \leq 1$, which yields $A_{\kappa} = (A_p)^{\kappa/p}, B_{\kappa} = (B_p)^{\kappa/p}$.

We prove the result separately for $p\in(1,2]$ and $p>2$. For $p\in(1,2]$, we apply the Almost Supermartingale Convergence Theorem to the transformed Lyapunov function $\Psi=\Phi^{p/2}$. For $p>2$, the direct argument introduces potentially non-summable second-order terms, so we instead apply an interval-projection argument to $u_k=\sqrt{\Phi(x_k-x^\star)}$. This controls upcrossings of $[D,2D]$ for every $D>0$ and ultimately shows that $u_k\to0$ almost surely.

\subsection{Case 1: \texorpdfstring{$p \in (1,2]$}{1 < p <= 2}}
In this case, we will show that if $\Phi$ is a Lyapunov function satisfying Assumption \ref{assumption: general-drift-condition} then $\Phi^{\frac{p}{2}}$ would satisfy similar drift and smoothness conditions. Thus, we first prove the following lemma:
\begin{lemma}
\label{lemma: suitable_drift}
Let $\Phi:\mathbb{R}^d\to\mathbb{R}_+$ satisfy Assumption~\ref{assumption: general-drift-condition}.
For $1<p\le2$, define $\Psi(x):=\Phi(x)^{p/2}$. Then there exist constants
$a_1=c_1^{p/2}$, $a_2=c_2^{p/2}$ and $\eta_p,L_p>0$ such that for all $x,y\in\mathbb{R}^d$,
\begin{align}
    \label{eqn: p-growth-bound-of-psi}
    a_1\|x\|^p \le \Psi(x) \le a_2\|x\|^p, \\
    \langle\nabla\Psi(x-x^\star),H(x)-x\rangle \le -\eta_p\Psi(x-x^\star), \\
    \label{eqn: p-smoothness}
    \Psi(x)\le \Psi(y)+\langle\nabla\Psi(y),x-y\rangle+\frac{L_p}{p}\|x-y\|^p.
\end{align}
\end{lemma}

\begin{proof}
From $c_1\|x\|^2\le\Phi(x)\le c_2\|x\|^2$, we have
\begin{align}
    c_1^{p/2}\|x\|^p \le \Psi(x) \le c_2^{p/2}\|x\|^p.
\end{align}
Moreover, by the chain rule and smoothness of $\Phi$ according to Assumption \ref{assumption: general-drift-condition}, we have
\begin{align}
    \nabla\Psi(x)=\frac{p}{2}\Phi(x)^{\frac{p}{2}-1}\nabla\Phi(x) \, \forall x \neq 0.
\end{align}
To prove that $\Psi$ is differentiable at zero with gradient zero, consider the candidate derivative $g = 0$. Then from \eqref{eqn: lyapunov-sandwich} and note that $\Psi(0) = 0$, we have for $p > 1$ that
\begin{align*}
    \frac{|\Psi(x) - \Psi(0) - \langle 0, x \rangle|}{\norm{x}_2} = \frac{\Psi(x)}{\norm{x}_2} \leq c_2^{p/2} \norm{x}_2^{p-1} \to 0
\end{align*}
as $x \to 0$. From Definition 3.28 on differentiability in \cite{beck-book}, it follows that $\Psi$ is differentiable at $0$ and $\nabla\Psi(0)=0$. Hence, we have
\begin{align}
    \langle\nabla\Psi(x-x^\star),H(x)-x\rangle \le -\tfrac{p}{2}\eta\,\Phi(x-x^\star)^{p/2} = -\eta_p\Psi(x-x^\star).
\end{align}
Choose $x=y-\frac{1}{L_2}\nabla\Phi(y)$. By smoothness of $\Phi$, we have
\begin{align}
0\le\Phi(x)\le\Phi(y)-\frac{1}{2L_2}\|\nabla\Phi(y)\|^2
\;\Rightarrow\;
\|\nabla\Phi(y)\|^2\le2L_2\Phi(y)\le2L_2c_2\|y\|^2,
\end{align}
and therefore,
\begin{align}
    \label{eqn: psi-gradient-growth}
    \|\nabla\Psi(y)\| \le \tfrac{p}{2}\Phi(y)^{\frac{p}{2}-1}\|\nabla\Phi(y)\| \le C\|y\|^{p-1}.
\end{align}
If $\|y\|\le\frac12\|x-y\|$, then $\|x\|\le \norm{y} + \norm{x-y} \le\frac32\|x-y\|$ and
\begin{align}
    \Psi(x) - \Psi(y) - \langle \nabla \Psi(y), x - y \rangle &\le \| \Psi(x) \| + \| \Psi(y) \| + \| \nabla \Psi(y) \| \, \| x - y \| \notag \\
    &\le a_2 \|x\|^p + a_2 \|y\|^p + C \|y\|^{p-1} \|x - y\| \notag \\
    &\le \underbrace{\left( 2c_2^{\frac{p}{2}} \left( \frac{3}{2} \right)^p + C \left( \frac{3}{2} \right)^{p-1} \right)}_{= C_1} \|x - y\|^p,
\end{align}
where the first inequality holds from the Cauchy-Schwarz inequality, the second inequality holds from \eqref{eqn: p-growth-bound-of-psi} and \eqref{eqn: psi-gradient-growth}, and the third inequality comes from the fact that $\max(\norm{x},\norm{y}) \leq \frac{3}{2} \norm{x-y}$.

If $\|y\| \geq \frac{1}{2}\|x-y\|$, then concavity of $t^{p/2}$ and smoothness of $\Phi$ yield
\begin{align}
    \Psi(x)-\Psi(y) &\overset{(a)}{\le} \tfrac{p}{2}\Phi(y)^{\frac{p}{2}-1}\big(\Phi(x)-\Phi(y)\big) \notag \\
    &\overset{(b)}{\le} \tfrac{p}{2}\Phi(y)^{\frac{p}{2}-1} \br{\langle \nabla \Phi(y), x-y \rangle + \tfrac{L_2}{2}\|x-y\|_2^2} \notag \\
    &\overset{(c)}{=}\langle \nabla \Psi(y), x-y \rangle + \tfrac{p L_2}{4}\Phi(y)^{\frac{p}{2}-1}\|x-y\|_2^2 \notag \\
    &\overset{(d)}{\le} \langle \nabla \Psi(y), x-y \rangle + \tfrac{p L_2}{4} c_1^{\frac{p}{2}-1} \|y\|_2^{p-2}\|x-y\|_2^2 \notag \\
    &\overset{(e)}{\le} \langle \nabla \Psi(y), x-y \rangle + C_2 \|y\|^{p-2}\|x-y\|_2^2 \notag \\
    &\overset{(f)}{\le} \langle \nabla \Psi(y), x-y \rangle + C_3 \|x-y\|_2^p.
\end{align}
Here, (a) comes from concavity of $t^\frac{p}{2}$ for $p \in (1,2]$, (b) follows from \eqref{eqn: holder-smoothness}, (c) follows from the definition of $\Psi$, (d) follows from \eqref{eqn: p-growth-bound-of-psi}, (e) and (f) follow from $\|y\| \geq \frac{1}{2}\|x-y\|$.  Taking $L_p=p\max\{C_1,C_3\}$ concludes the proof.
\end{proof}

Next, fix $1<\kappa\leq \min\{p,2\}$ and choose $\Psi(x) = \Phi(x)^{\kappa/2}$ as the Lyapunov function. We establish the following drift inequality.
\begin{lemma}
\label{lemma: p_smaller_2_general_drift}
There exists a deterministic index $K_\kappa < \infty$ such that, for every $k\geq K_\kappa$ and $1 < \kappa \leq \min\{p,2\}$, we have
\begin{align}
    \E\!\sqbr{\Psi(x_{k+1}-x^\star) \mid \cF_k} \leq \left(1-\frac{\eta_\kappa}{2}\alpha_k\right) \Psi(x_k-x^\star) + \frac{L_\kappa}{\kappa} 2^{\kappa-1}A_\kappa\alpha_k^\kappa.
\end{align}
\end{lemma}

\begin{proof}
Using Lemma~\ref{lemma: suitable_drift} with $x=x_{k+1}-x^\star$ and $y=x_k-x^\star$, we obtain
\begin{align}
    \Psi(x_{k+1}-x^\star) &\leq \Psi(x_k-x^\star) + \left\langle \nabla\Psi(x_k-x^\star), x_{k+1}-x_k\right\rangle + \frac{L_\kappa}{\kappa} \|x_{k+1}-x_k\|^\kappa.
\end{align}
We next bound the expectation of the remainder term $\norm{x_{k+1}-x_k}^\kappa = \alpha_k^\kappa \norm{H(x_k)-x_k+w_k}^\kappa$. Using Holder's inequality, we have
\begin{align}
    \label{equation: sub_quadratic_tail_estimation}
    \E\!\sqbr{\|H(x_k)-x_k+w_k\|^\kappa \mid \cF_k} \leq 2^{\kappa-1} \|H(x_k)-x_k\|^\kappa + 2^{\kappa-1} \E\!\sqbr{\|w_k\|^\kappa \mid \cF_k}.
\end{align}
It follows that $\E\!\sqbr{\Psi(x_{k+1}-x^\star) \mid \cF_k}$ can be upper bounded as
\begin{align*}
    &\overset{(a)}{\leq} \Psi(x_k-x^\star) + \E\!\sqbr{\left\langle \nabla\Psi(x_k-x^\star), x_{k+1}-x_k \right\rangle \mid\cF_k} + \frac{L_\kappa}{\kappa} \E\!\sqbr{\|x_{k+1}-x_k\|^\kappa \mid \cF_k} \\
    &\overset{(b)}{=} \Psi(x_k-x^\star) + \alpha_k \E\!\sqbr{\left\langle \nabla\Psi(x_k-x^\star), H(x_k)-x_k+w_k \right\rangle \mid\cF_k} + \frac{L_\kappa}{\kappa}\alpha_k^\kappa \E\!\sqbr{\|H(x_k)-x_k+w_k\|^\kappa \mid\cF_k}
    \\
    &\overset{(c)}{\leq} \Psi(x_k-x^\star) +
    \alpha_k \left\langle \nabla\Psi(x_k-x^\star), H(x_k)-x_k \right\rangle + \frac{L_\kappa}{\kappa} 2^{\kappa-1} \alpha_k^\kappa \br{\|H(x_k)-x_k\|^\kappa + \E\!\sqbr{\|w_k\|^\kappa \mid \cF_k}}
    \\
    &\overset{(d)}{\leq} \Psi(x_k-x^\star) - \eta_\kappa\alpha_k\Psi(x_k-x^\star) + \frac{L_\kappa}{\kappa} 2^{\kappa-1} \alpha_k^\kappa \|H(x_k)-x_k\|^\kappa + \frac{L_\kappa}{\kappa} 2^{\kappa-1} \alpha_k^\kappa \E\!\left[\|w_k\|^\kappa \mid\cF_k\right] \\
    &\overset{(e)}{\leq} (1-\eta_\kappa\alpha_k)\Psi(x_k-x^\star) + \frac{L_\kappa}{\kappa} 2^{\kappa-1} (C+1)^\kappa\alpha_k^\kappa \|x_k-x^\star\|^\kappa + \frac{L_\kappa}{\kappa} 2^{\kappa-1} \alpha_k^\kappa \br{A_\kappa + B_\kappa\|x_k-x^\star\|^\kappa} \\
    &\overset{(f)}{\leq} \br{1-\eta_\kappa\alpha_k + \frac{L_\kappa}{a_1\kappa} 2^{\kappa-1} \br{(C+1)^\kappa + B_\kappa} \alpha_k^\kappa} \Psi(x_k-x^\star) + \frac{L_\kappa}{\kappa} 2^{\kappa-1}A_\kappa\alpha_k^\kappa.
\end{align*}
Here, (a) follows from \eqref{eqn: p-smoothness} in Lemma~\ref{lemma: suitable_drift}, (b) follows from \eqref{equation: main_iterate}, (c) follows from Assumption~\ref{assumption: noise-unbiasedness} and \eqref{equation: sub_quadratic_tail_estimation}, (d) follows from the drift property of $\Psi$ in Lemma~\ref{lemma: suitable_drift}, (e) follows from Assumptions~\ref{assumption: Lipschitz} and \ref{assumption: noise-condition-p}; and (f) follows from \eqref{eqn: p-growth-bound-of-psi}. Since $\kappa>1$ and $\alpha_k\to0$, there exists a deterministic index $K<\infty$ such that, for every $k\geq K$,
\begin{align*}
    -\eta_\kappa\alpha_k + \frac{L_\kappa}{a_1\kappa} 2^{\kappa-1} \br{(C+1)^\kappa + B_\kappa} \alpha_k^\kappa \leq -\frac{\eta_\kappa}{2}\alpha_k.
\end{align*}
Substituting this inequality into the preceding bound proves the claim.
\end{proof}

Back to the proof of Theorem \ref{theorem: main-result} for the case $p \in (1,2]$. From Lemma \ref{lemma: p_smaller_2_general_drift} and apply Theorem \ref{theorem: sct_theorem} for $Y_k = \Psi\br{x_k-x^\star}, X_k = \frac{\eta_p}{2} \alpha_k \Psi\br{x_k-x^\star}, Z_k = \frac{L_p}{p} 2^{p-1} A_p \alpha_k^p$, we have $\Psi\br{x_k-x^\star}$ converges almost surely and $\sum_k \alpha_k \Psi(x_k-x^\star) < \infty$. Since $\sum_{k=0}^{\infty} \alpha_k = \infty$, we must have $\Psi(x_k-x^\star) \rightarrow 0$ almost surely. Indeed, if this limit were strictly positive, then following a similar argument as in \cite{zaiwei-envelope}, $\Psi(x_k-x^\star)$ would be bounded away from zero for all sufficiently large $k$, which is a contradiction. From \eqref{eqn: lyapunov-sandwich}, we have $\norm{x_k-x^\star} \to 0$ almost surely. Since $\R^d$ is a normed space, we have $x_k \to x^\star$ almost surely, as desired.

\subsection{Case 2: \texorpdfstring{$p \in (2,\infty)$}{p > 2}}
Let us define $u_k = \sqrt{\Phi(x_k - x^\star)}$, we wish to first control the magnitude of $|u_{k+1}-u_k|$:
\begin{lemma}
\label{lemma: uk-one-step-absolute-drift}
There exists a constant $b_3$ such that
\begin{align}
\label{equation: almost_surely_drift_distance}
    |u_{k+1} - u_k| \le b_3 \alpha_k (u_k + \|w_k\|), \qquad \text{a.s.}
\end{align}
\end{lemma}

\begin{proof}
We first prove that $\sqrt{\Phi}$ is globally Lipschitz, including at the
origin.  By the quadratic growth condition \eqref{eqn: lyapunov-sandwich},
$\Phi(0)=0$.  Since $\Phi$ is differentiable and nonnegative, the origin is
a global minimizer of $\Phi$, and hence $\nabla\Phi(0)=0$. For any $z\in\R^d$, apply the smoothness inequality
\eqref{eqn: holder-smoothness} with $y=z-\frac{1}{L_2}\nabla\Phi(z)$ and from the nonnegativity of $\Phi$, we obtain
\begin{align*}
    0 &\leq \Phi\left(z-\frac{1}{L_2}\nabla\Phi(z)\right) \leq \Phi(z)-\frac{1}{2L_2}\|\nabla\Phi(z)\|_2^2.
\end{align*}
Consequently, this gives
\begin{align}
\label{equation: phi-self-bounding-gradient}
    \|\nabla\Phi(z)\|_2^2\leq 2L_2\Phi(z),
    \qquad \forall z\in\R^d.
\end{align}

Now fix arbitrary $x,y\in\R^d$.  First suppose that the line segment$[x,y]$ does not contain the origin.  The quadratic lower bound in \eqref{eqn: lyapunov-sandwich} implies that $\Phi(z)>0$ for every $z\in[x,y]$.  Therefore, $\sqrt{\Phi}$ is differentiable at every point of $[x,y]$, and \eqref{equation: phi-self-bounding-gradient} gives
\begin{align}
    \|\nabla\sqrt{\Phi}(z)\|_2 = \frac{\|\nabla\Phi(z)\|_2}{2\sqrt{\Phi(z)}} \leq \sqrt{\frac{L_2}{2}}, \qquad z\in[x,y].
\end{align}
Apply the one-dimensional mean value theorem to
$g(t):=\sqrt{\Phi(x+t(y-x))}$ on $[0,1]$.  The preceding gradient bound
yields
\begin{align}
    \big|\sqrt{\Phi(x)}-\sqrt{\Phi(y)}\big|
    \leq \sqrt{\frac{L_2}{2}}\,\|x-y\|_2.
\end{align}
It remains to consider the case $0\in[x,y]$, which includes $x=0$ or
$y=0$.  Applying \eqref{eqn: holder-smoothness} with $x=0$, and using
$\Phi(0)=0$ and $\nabla\Phi(0)=0$, gives
\begin{align}
    \sqrt{\Phi(v)}\leq \sqrt{\frac{L_2}{2}}\,\|v\|_2,
    \qquad \forall v\in\R^d.
\end{align}
Since $0\in[x,y]$, we have $\|x-y\|_2=\|x\|_2+\|y\|_2$.  Hence,
\begin{align*}
    \big|\sqrt{\Phi(x)}-\sqrt{\Phi(y)}\big| &\leq \sqrt{\Phi(x)}+\sqrt{\Phi(y)} \leq \sqrt{\frac{L_2}{2}} \left(\|x\|_2+\|y\|_2\right) =\sqrt{\frac{L_2}{2}}\,\|x-y\|_2.
\end{align*}
Thus, in all cases, we have
\begin{align}
\label{equation: gap_bounding_eqn_1}
    \big|\sqrt{\Phi(x)}-\sqrt{\Phi(y)}\big| \leq \sqrt{\frac{L_2}{2}}\|x-y\|_2.
\end{align}
On the other hand, by Assumption \ref{assumption: Lipschitz},
we observe that:
\begin{align*}
    \|H(x_k) - x_k\| \le \|H(x_k) - x^\star\| + \|x_k - x^\star\| \le (C+1)\|x_k - x^\star\|.
\end{align*}
Using the Stochastic Approximation update and the preceding estimate, we get
\begin{align}
    \|x_{k+1}-x_k\|_2 &\leq \alpha_k\left(\|H(x_k)-x_k\|_2+\|w_k\|_2\right) \notag \\
    \label{equation: gap_bounding_eqn_2}
    &\leq \alpha_k\left((C+1)\|x_k-x^\star\|_2+\|w_k\|_2\right),
\qquad \text{a.s.}
\end{align}
Using both estimates \eqref{equation: gap_bounding_eqn_1} and \eqref{equation: gap_bounding_eqn_2},
together with the lower quadratic bound in
\eqref{eqn: lyapunov-sandwich}, we have
\begin{align*}
    |u_{k+1}-u_k| &\leq \sqrt{\frac{L_2}{2}}\|x_{k+1}-x_k\|_2 \\
    &\leq \sqrt{\frac{L_2}{2}}\alpha_k \left(\frac{C+1}{\sqrt{c_1}}u_k+\|w_k\|_2\right) \\
    &\leq b_3\alpha_k\left(u_k+\|w_k\|_2\right), \qquad \text{a.s.},
\end{align*}
where $b_3=\sqrt{\frac{L_2}{2}} \max\left\{\frac{C+1}{\sqrt{c_1}},1\right\}$. Hence proved.
\end{proof}
For an estimation of the polynomial potential function $|x|^p$,
we use the following scalar version of Lemma~2.5 in~\citep{adil2024fast}:

\begin{lemma}
\label{lemma: p_larger_2_lp_regression_drift}
For any $x,\Delta \in \mathbb{R}$ and any $p \ge 2$,
\begin{align*}
    \frac{p}{8}\, |x|^{p-2} \Delta^{2} + \frac{1}{2^{p+1}} |\Delta|^{p} \le |x+\Delta|^{p} - |x|^{p} - p |x|^{p-1} \operatorname{sgn}(x)\,\Delta \le 2 p^{2}\, |x|^{p-2} \Delta^{2} + p^{p} |\Delta|^{p}.
\end{align*}
\end{lemma}

Back to the main proof, for this case,  utilizing the Lyapunov function $\Phi(x)$ or equivalently using Lemma \ref{lemma: p_smaller_2_general_drift} with $\kappa = 2$, there exist constants $b_1,b_2$ such that:
\begin{align}
\label{equation: p_larger_2_general_second_moment_drift}
    \E[\Phi(x_{k+1} - x^\star)\mid \cF_k] \le (1-b_1\alpha_k)\Phi(x_k - x^\star) + b_2\alpha_k^2.
\end{align}
In particular, Lemma \ref{lemma: p_smaller_2_general_drift} with $\kappa = 2$ says that we can choose $b_1 = \frac{\eta_2}{2}, b_2 = \frac{L_2}{2} A_2$. Applying the conditional Jensen inequality, we easily get:
\begin{align}
    \E\sqbr{\sqrt{\Phi(x_{k+1} - x^\star)}\mid \cF_k} &\le \sqrt{(1-b_1\alpha_k)\Phi(x_k - x^\star) + b_2\alpha_k^2} \notag \\
    \label{eqn: sqrt-quadratic-drift}
    \Leftrightarrow \E[u_{k+1} | \cF_k] &\le \sqrt{(1-b_1\alpha_k)u_k^2 + b_2\alpha_k^2}.
\end{align}
Let
\begin{align*}
    z_k &= \max(0,\min(2D,u_k)-D) = \max(D,\min(2D,u_k)) - D \in [0,D],
\end{align*}
where $D>0$ is an arbitrary positive constant. We first show that $z_k$ converges almost surely, and then show that this limit must be $0$. Indeed, we aim to show that:
\begin{lemma}
\label{lemma: zk-drift}
There exist a constant $b_4 \geq 0$ and a deterministic index $K_D<\infty$ such that, for every $k \geq K_D$, we have
\begin{align}
    \label{equation: p_larger_rare_event_drift}
    \E[z_{k+1}^p \mid \cF_k] \le z_k^p + b_4 \alpha_k^p.
\end{align}
\end{lemma}

\begin{proof}
To this end, we consider the following cases:
\begin{case}
\label{case: contraction}
Contraction phase $u_k \geq 2D$.
\end{case}
In this case, we have $z_k = D$. Since $z_{k+1} \le \min(2D,u_k) - D \leq D$ almost surely, we have
\begin{align*}
    \E[z_{k+1}^p \mid \cF_k] \le z_k^p.
\end{align*}

\begin{case}
\label{case: fluctuation}
Fluctuation phase $D \le u_k < 2D$.
\end{case}
Here, we have $z_k = u_k - D$. Observe that from $u_k \in [D,2D)$ and the definition of $z_k$, we have:
\begin{align*}
    0 \leq z_k = \max(D, \min(2D,u_k)) - D \leq | \min(2D,u_k) - D|.
\end{align*}
Applying Lemma~\ref{lemma: p_larger_2_lp_regression_drift} with
$x=z_k$ and $\Delta=\min(2D,u_{k+1})-u_k$, we obtain
\begin{align*}
    \E[z_{k+1}^p \mid \cF_k] &\le \E\!\left[|\min(2D,u_{k+1})-D|^p \mid \cF_k\right] \\
    &= \E\!\left[\big|\underbrace{\min(2D,u_{k+1})-u_k}_{= \Delta}+z_k\big|^p \mid \cF_k\right] \\
    &\le z_k^p + p z_k^{p-1} \E[\Delta \mid \cF_k] + 2p^2 z_k^{p-2} \E[\Delta^2 \mid \cF_k] + p^p \E[|\Delta|^p \mid \cF_k].
\end{align*}
Since $u_k<2D$, we have $\min(2D,u_k)=u_k$. The map
$u\mapsto\min(2D,u)$ is non-expansive, and hence
\begin{align*}
    |\Delta| &= |\min(2D,u_{k+1})-\min(2D,u_k)|
    \leq |u_{k+1}-u_k|.
\end{align*}
Moreover, observe that we also have $\Delta \leq u_{k+1}-u_k$. Combine this with Lemma \ref{lemma: uk-one-step-absolute-drift} gives
\begin{align}
    \E[z_{k+1}^p \mid \cF_k] &\le z_k^p + p z_k^{p-1} \E[u_{k+1}-u_k \mid \cF_k] \notag\\
    &\quad + 2p^2 z_k^{p-2} b_3^2 \alpha_k^2 \E[(u_k+\|w_k\|)^2 \mid \cF_k] \notag\\
    &\quad + (b_3 p)^p \alpha_k^p \E[(u_k+\|w_k\|)^p \mid \cF_k].
\end{align}
Since $D\leq u_k<2D$, conditional Lyapunov's inequality, Assumption \ref{assumption: noise-condition-p}, and the lower bound in \eqref{eqn: lyapunov-sandwich} yield
\begin{align*}
    \E[(u_k+\|w_k\|)^2\mid\cF_k]
    &\leq 2u_k^2+2\E[\|w_k\|^2\mid\cF_k] \\
    &\leq 2u_k^2+2\bigl(\E[\|w_k\|^p\mid\cF_k]\bigr)^{2/p} \\
    &\leq 2u_k^2+2A_p^{2/p}
    +2B_p^{2/p}\|x_k-x^\star\|^2 \\
    &\leq 2\left((2D)^2+A_p^{2/p}
    +\frac{B_p^{2/p}}{c_1}(2D)^2\right),
\end{align*}
where $b_3$ is defined as in Lemma \ref{lemma: uk-one-step-absolute-drift}. Similarly, we also have
\begin{align*}
    \E[(u_k+\|w_k\|)^p\mid\cF_k] &\leq 2^{p-1}\left(u_k^p+\E[\|w_k\|^p\mid\cF_k]\right) \\
    &\leq 2^{p-1}\left(u_k^p+A_p+B_p\|x_k-x^\star\|^p\right) \\
    &\leq 2^{p-1}\left((2D)^p+A_p +\frac{B_p}{c_1^{p/2}}(2D)^p\right).
\end{align*}
Substituting these estimates above, we obtain
\begin{align}
    \E[z_{k+1}^p \mid \cF_k] &\le z_k^p + p z_k^{p-1} \E[u_{k+1}-u_k \mid \cF_k] \notag\\
    &\quad + z_k^{p-2} \alpha_k^2 \underbrace{4b_3^2 p^2
    \left((2D)^2+A_p^{2/p}+\frac{B_p^{2/p}}{c_1}(2D)^2\right)}_{=C_1} \notag\\
    &\quad + \alpha_k^p \underbrace{(2 b_3 p)^p \left((2D)^p+A_p+\frac{B_p}{c_1^{p/2}}(2D)^p\right)}_{=C_2} \notag\\
    &= z_k^p + \underbrace{p z_k^{p-1} \E[u_{k+1}-u_k \mid \cF_k] + C_1 z_k^{p-2} \alpha_k^2}_{T_{\mathrm{mid}}} + C_2 \alpha_k^p.
\end{align}

It remains to show that $T_{\mathrm{mid}} = \mathcal O(\alpha_k^p)$. To this end, we want to show that the term $\E[u_{k+1}-u_k \mid \cF_k]$ admits a strong enough negative drift so that it can dominate the positive term $C_1 z_k^{p-2} \alpha_k^2$. Observe that $z_k$ is bounded from the projection step, it is sufficient to show that $\E[u_{k+1}-u_k \mid \cF_k] \leq -\Omega(\alpha_k)$ so that it can dominate the $\alpha_k^2$ in the positive term. Indeed, we have the following observation.

\begin{claim}
\label{claim: uk-negative-drift-alpha-k}
There exists a deterministic index $K_D<\infty$ such that, for every
$k\geq K_D$, on the event $\{D\leq u_k<2D\}$, we have
\begin{align}
    \E[u_{k+1}-u_k \mid \cF_k]
    \leq -\frac{b_1D}{4}\alpha_k,
\end{align}
where $b_1 = \frac{\eta_2}{2}$.
\end{claim}

\begin{proof}
From \eqref{eqn: sqrt-quadratic-drift}, we have
\begin{align}
    \E[u_{k+1}-u_k \mid \cF_k] \leq \sqrt{(1-b_1\alpha_k)u_k^2 + b_2\alpha_k^2} - u_k.
\end{align}
Since $\alpha_k\to 0$, there exists a deterministic index $K_D<\infty$
such that, for every $k\geq K_D$, we have
\begin{align}
    b_1\alpha_k\leq 1
    \qquad\text{and}\qquad
    b_2\alpha_k\leq \frac{b_1D^2}{2}.
\end{align}
Fix $k\geq K_D$ and suppose that $D \leq u_k < 2D$, we have $(1-b_1\alpha_k)u_k^2+b_2\alpha_k^2 \geq 0$ and
\begin{align}
    (1-b_1\alpha_k)u_k^2+b_2\alpha_k^2-u_k^2 &=-b_1\alpha_k u_k^2+b_2\alpha_k^2 \leq -b_1\alpha_k u_k^2+\frac{b_1}{2}\alpha_kD^2 \leq -\frac{b_1}{2}\alpha_k u_k^2.
\end{align}
Since $u_k\geq D>0$, rationalizing the
difference of square roots yields
\begin{align}
    \sqrt{(1-b_1\alpha_k)u_k^2+b_2\alpha_k^2}-u_k
    &=\frac{(1-b_1\alpha_k)u_k^2+b_2\alpha_k^2-u_k^2}{\sqrt{(1-b_1\alpha_k)u_k^2+b_2\alpha_k^2}+u_k} \notag\\
    &\leq -\frac{b_1\alpha_k u_k^2}{2(\sqrt{(1-b_1\alpha_k)u_k^2+b_2\alpha_k^2}+u_k)} \notag\\
    &\leq -\frac{b_1}{4}\alpha_k u_k
    \leq -\frac{b_1D}{4}\alpha_k,
\end{align}
where the penultimate inequality uses $\sqrt{(1-b_1\alpha_k)u_k^2+b_2\alpha_k^2} \leq u_k$ and the last inequality uses $u_k \geq D$. Hence proved.
\end{proof}

Back to the main problem, let $c_D=\frac{b_1D}{4}$ and choose $b_{5,D}=\frac{C_1}{pc_D}=\frac{4C_1}{pb_1D}$. Then, whenever $z_k>b_{5,D}\alpha_k$, we have
\begin{align}
    p c_D\alpha_k z_k^{p-1} \ge C_1 z_k^{p-2}\alpha_k^2,
\end{align}
and so $T_{\text{mid}} \leq 0$. Otherwise, if $z_k\leq b_{5,D}\alpha_k$, we obtain
\begin{align}
    -p c_D\alpha_k z_k^{p-1} + C_1 z_k^{p-2}\alpha_k^2 \le C_1 z_k^{p-2}\alpha_k^2 \le C_1 b_{5,D}^{p-2} \alpha_k^p.
\end{align}
Absorbing constants yields
\begin{align}
    \E[z_{k+1}^p \mid \cF_k] \le z_k^p + b_{4,1} \alpha_k^p.
\end{align}

\begin{case}
\label{case: stabilization}
Stabilization phase $0 \le u_k < D$.
\end{case}
In this case, we have $z_k = 0$, and from Lemma \ref{lemma: uk-one-step-absolute-drift}, we have
\begin{align}
    z_{k+1} = z_{k+1} - z_k \le |u_{k+1}-u_k| \le b_3 \alpha_k (D+\|w_k\|), \qquad \text{a.s.}
\end{align}
Here, since $z_k+D$ is the projection of $u_k$ onto $[D,2D]$, $z_{k+1}-z_k = (z_{k+1}+D)-(z_k+D) \leq |u_{k+1}-u_k|$ holds from the non-expansiveness of the projection operator \citep{beck-book}. Therefore, since $z_k = 0$ and from the previous inequality as well as Holder's inequality and \eqref{eqn: lyapunov-sandwich}, we have
\begin{align}
    \E[z_{k+1}^p \mid \cF_k] &\le \underbrace{(b_3)^p 2^{p-1} \br{D^p + A_p + \frac{B_p}{c_1^{p/2}} D^p}}_{= b_{4,2}}\alpha_k^p = z_k^p + b_{4,2} \alpha_k^p .
\end{align}
By letting $b_4 = \max(b_{4,1},b_{4,2})$, the proof is complete.
\end{proof}

Applying the Supermartingale Convergence Theorem \ref{theorem: sct_theorem}
to the tail sequence $(z_k^p)_{k\geq K_D}$ in Lemma
\ref{lemma: zk-drift}, we get that $z_k^p$, and hence $z_k$, converges
almost surely.

Next, we show that $z_k$ converges almost surely to $0$. First, note that since $\sum \alpha_k^p < \infty$, we must have $\lim \alpha_k = 0$, so there exists $K$ such that for all $k \geq K$, we have $\alpha_k < 1$. Let $s(x) = \sum_{k \geq K} \alpha_k^x$, it's clear that $s(x)$ is a decreasing function with respect to $x$. From here, we have the following claim:

\begin{claim}
\label{claim: s(q) finite, s(q-1) infty}
There exists a value $q \in [2,p]$ such that $s(q-1) = \infty $ and $s(q) < \infty$.
\end{claim}

\begin{proof}
From the step size condition \eqref{eqn: step-size-condition}, we have $s(1) = \infty$ and $s(p) < \infty$. Suppose the claim is false, then we have $s(1) = s(2) = \infty$. On the other hand, we have $s(p) < \infty$, thus, we must have $s(p-1) < \infty$ also; otherwise, $q = p$ is the value we seek. By repeating this argument, we get $s(p - \lfloor p \rfloor + 1) < \infty$. However, since $1 \leq p - \lfloor p \rfloor + 1 < 2$ by definition, we must have $s(2) < \infty$, which is a contradiction. Thus, there must exist a $q \geq 2$ such that $s(q-1) = \infty$ and $s(q) < \infty$.
\end{proof}
Now, by \eqref{equation: p_larger_2_general_second_moment_drift}, we have
\begin{align}
    \E[u_{k+1}^2 | \cF_k] \leq (1-b_1\alpha_k)u_k^2 + b_2\alpha_k^2.
\end{align}
This implies:
\begin{align}
    \E[\alpha_{k+1}^{q-2}u_{k+1}^2 | \cF_k] \leq \alpha_k^{q-2}u_k^2 - b_1\alpha_k^{q-1}u_k^2 + b_2\alpha_k^q.
\end{align}
By Supermartingale Convergence Theorem \ref{theorem: sct_theorem} and from the choice of $q$ in Claim \ref{claim: s(q) finite, s(q-1) infty} so that $\sum \alpha_k^q < \infty$, we have $\sum \alpha_k^{q-1}u_k^2 < \infty$ almost surely. Here, note that we do not have $u_k$ converging almost surely yet, and so we have to prove its almost sure convergence by showing $z_k$ converges to $0$ almost surely.

Indeed, suppose the contrary that $\lim z_k = L$ almost surely for some constant $L > 0$ for some sample path of $x_k$, there exists constant $\epsilon > 0$ and $K_\epsilon$ such that for all $k \geq K_\epsilon$, we have $z_k = \max(\min(2D,u_k),D) - D > \varepsilon$. Thus, for all $k \geq K_\epsilon$, we must have $u_k > D + \varepsilon$. However, since $\sum \alpha_k^{q-1} = \infty$ and $u_k > L$ for all $k \geq \max(K,K_\epsilon)$, we have $\sum \alpha_k^{q-1}u_k^2 = \infty$, contradiction. Thus, $z_k$ converges almost surely to 0 and therefore $\limsup_{k \rightarrow \infty} u_k \leq D $ almost surely.

To complete the argument, we will show $u_k \rightarrow 0$ almost surely as $k \rightarrow \infty$. To this end, consider the countable sequence $D_m = \frac{1}{m}$ for $m \in \Z^+$ and define
\begin{align*}
    \Omega_m :=
    \left\{\omega: \limsup_{k\to\infty}u_k(\omega) \leq \frac{1}{m}\right\}.
\end{align*}
Applying the preceding argument with the fixed deterministic value $D_m=1/m$ gives $\P[\Omega_m] = 1\, \forall m \in \Z^+$. Taking the countable intersection
\begin{align*}
    \Omega^\star = \bigcap_{m=1}^{\infty}\Omega_m,
\end{align*}
since this is a countable intersection of probability $1$ events, we have $\P[\Omega^\star] = 1$. Fix any $\omega \in \Omega^\star$, we have
\begin{align*}
    \limsup_{k\to\infty}u_k(\omega) \leq \frac{1}{m}.
\end{align*}
Taking $m \rightarrow \infty$ on both sides and notice that $u_k \geq 0$ give
\begin{align*}
    0 \leq \liminf_{k\to\infty}u_k(\omega) \leq \limsup_{k\to\infty}u_k(\omega) \leq 0,
\end{align*}
which implies $u_k \rightarrow 0$ almost surely. Finally, we prove the uniqueness of the solution: Indeed, let $x^\star, y^\star$ be fixed points of \eqref{eqn: fixed-point-equation}, from Assumption \ref{assumption: general-drift-condition}, we have that
\begin{align}
    0 = \langle \nabla \Phi(y^\star-x^\star), \underbrace{H(y^\star)-y^\star}_{= 0} \rangle \leq -\eta \Phi(y^\star-x^\star).
\end{align}
However, since $\Phi \geq 0$ and $\eta > 0$, we must have that $\Phi(y^\star-x^\star) = 0$. On the other hand, by Equation \eqref{eqn: lyapunov-sandwich}, we have $0 \geq c_1 \|x^{\star} - y^{\star}\|^2$, which implies $x^{\star} = y^{\star}$. Thus, since $\Phi(x_k - x^{\star}) \rightarrow 0$ almost surely, we must have $x_k$ converge to the unique solution of \eqref{eqn: fixed-point-equation}.

\end{proof}

\begin{remark}
Our projection technique has some resemblances to the recent work of \citep{liu2025extensions}, where the author shows that SA iterates satisfying mild conditions converge to a bounded interval. Nevertheless, our approach departs from \citep{liu2025extensions} in several crucial ways that are critical to the success of the analysis. In particular, \citep{liu2025extensions} assumes that the associated Lyapunov sequence $\{u_k\}$ satisfies $|u_{k+1}-u_k| \le D \alpha_k (u_k + 1)$ almost surely for some constants $D > 0$, an assumption that is generally violated in the presence of multiplicative noise. In particular, an arbitrarily large noise realization may occur with nonzero probability, precluding the existence of such a bound for a fixed $D > 0$. Moreover, the proof in \citep{liu2025extensions} relies on the fact that for sufficiently large $u_k > D_3$, the conditional drift satisfies $\mathbb{E}[u_{k+1} - u_k \mid \mathcal{F}_k] \leq 0$, allowing the negative first-order term to dominate higher-order effects; however, this argument only holds in sufficiently large $D_3$ and therefore guarantees convergence merely to a bounded neighborhood.

Instead of projecting to $[D, \infty)$, our proof uses a projection to the interval $[D,2D]$, thus conveniently controls the growth gap between consecutive term $u_{k},u_{k+1}$ effectively and exploit the negative drift of the form $\E[u_{k+1} - u_k | \cF_k] \leq -\Omega(\alpha_k)$. We believe that such a technique, in combination with the use of the high-moment inequality, can be of independent interest to the literature.
\end{remark}

\section{Proof for Theorem \ref{theorem: nonexpansive-a.s.-convergence}}
\label{sec: non-expansive-proof-heavy-tailed-regime}

In the non-expansive case, we do not have a strong negative drift condition. The lack of this strong negative drift makes it very tricky to show $\norm{x_k-x^\star}$ eventually converges almost surely to $0$. And so, we first need to lower bound the drift when $x$ is some distance away from $x^\star$. To this end, we establish the following lemma, which is similar to the bounded gradient sublevel set (Assumption 3 in \citep{mertikopoulos2020almost}) or coercivity and the bounded gradient near $x^\star$ assumption (Assumption 3.1 (c) and 3.1 (d) in \citep{jin2025stochastic}).

\begin{lemma}
\label{lemma: H(x)-x bound when x is far from x*}
Let $H: \R^d \to \R^d$ be a non-expansive mapping such that its fixed-point set $\cX = \{ x \in \R^d \mid H(x) = x \}$ is bounded and non-empty. For any given $c > 0$, define the feasible domain $S_c$ as
\begin{align*}
S_c = \left\{ x \in \R^d \;\middle|\; \inf_{x^\star \in \cX} \|x - x^\star\| \ge c \right\}.
\end{align*}
Then, there exists an $\epsilon > 0$ such that for all $x \in S_c$, the fixed-point residual satisfies $\|x - H(x)\| \ge \epsilon$.
\end{lemma}

\begin{proof}
Let $A := \operatorname{Id} - H$ where $\operatorname{Id}$ is the identity operator. Since $H$ is non-expansive on $\R^d$, it is continuous, and hence $A$ is also continuous. Moreover, the zero set of $A$ coincides with the fixed-point set of $H$, that is,
\begin{align*}
A^{-1}(0) = \{x \in \R^d : H(x) = x\} = \cX.
\end{align*}
Next, we show that $A$ is monotone. For any $x,y \in \R^d$, by the Cauchy--Schwarz inequality and the non-expansiveness of $H$, we have
\begin{align*}
    \langle H(x) - H(y),\, x - y \rangle
    &\leq \|H(x) - H(y)\| \, \|x - y\| \leq \|x - y\|^2.
\end{align*}
Rearranging the above inequality yields
\begin{align*}
    \langle (x - H(x)) - (y - H(y)),\, x - y \rangle \geq 0,
\end{align*}
which shows that $A$ is monotone. Finally, by Exercise 2.4 in \citep{ryu2022large}, any continuous monotone operator is maximally monotone. Therefore, $A$ is maximally monotone.

Because $\cX$ is bounded, Theorem 1 in \citep{rockafellar1969local} guarantees that $0 \in \operatorname{int}(\operatorname{Im}(A))$. This implies that the inverse operator $A^{-1}$ is locally bounded at $0$. Therefore, there exist constants $R, \delta > 0$ such that the residual is bounded away from zero at infinity:
\begin{align*}
    \|x\| > R \implies \|A(x)\| > \delta.
\end{align*}
Next, consider the bounded region of the feasible space, defined by the set $K = S_c \cap \overline{B}(0, R)$. If $K=\emptyset$, then every $x\in S_c$ satisfies $\|x\|>R$, and therefore $\|x-H(x)\|=\|A(x)\|>\delta$. Thus, the lemma holds for $\epsilon = \delta$ in this case. Now, assume that $K\neq\emptyset$. As the intersection of closed sets and a bounded ball in $\R^d$, $K$ is compact. Since $x \mapsto \|A(x)\|$ is continuous and $K$ strictly excludes the zero set $\cX$, the Extreme Value Theorem guarantees a strictly positive minimum on $K$:
\begin{align*}
    m = \min_{x \in K} \|A(x)\| > 0.
\end{align*}
Let $\epsilon = \min\{\delta, m\} > 0$. For any $x \in S_c$, there are only two possibilities: either $\|x\| > R$, yielding $\|x - H(x)\| > \delta \ge \epsilon$, or $x \in K$, yielding $\|x - H(x)\| \ge m \ge \epsilon$. In all cases, $\|x - H(x)\| \ge \epsilon$, completing the proof.
\end{proof}

\begin{remark}
While not explicitly analyzed in this work, Lemma \ref{lemma: H(x)-x bound when x is far from x*} suggests that our analysis can also be applied in the stochastic nonconvex optimization setting studied in \citep{mertikopoulos2020almost,jin2025stochastic}. Indeed, as previously mentioned, our Lemma \ref{lemma: H(x)-x bound when x is far from x*} corresponds to the bounded gradient sublevel set assumption used in these works and the Lipschitz assumption (Assumption \ref{assumption: Lipschitz}) corresponds to the Lipschitz gradient assumption. On the other hand, our work does not use the bounded sublevel set condition in \cite{mertikopoulos2020almost} and the global/local boundedness of the $2p-2$-th moment of the noise. Thus, our non-expansive Stochastic Approximation setting can be considered as a generalization of the setting studied in these works.
\end{remark}

Using this lemma, the proof for non-expansive operators is more or less the same as when we have the negative drift Assumption \ref{assumption: general-drift-condition} as the drift term now becomes $\norm{H(x_k)-x_k}^2$ instead of $\Phi(x-x^\star)$.

Similar to the proof of Theorem \ref{theorem: main-result}, we again split the proof into $2$ regimes, $p \in (1,2]$ and $p \in (2,\infty)$. The proof for the regime $p \in (1,2]$ is in the same spirit as Theorem \ref{theorem: main-result}. On the other hand, for the regime $p \in (2,\infty)$, establishing that $u_k$ converges almost surely to $0$ directly is difficult, and so we also use the iterate projection technique to first show that the proxy variable $z_k$ converges almost surely to $0$ and then show that $\norm{x_k-x^\star}$ converges almost surely to  $0$.

Before we proceed with the proof, we have the following claim on the multiplicative noise.

\begin{claim}
Suppose Assumption \ref{assumption: noise-condition-p} holds for $p > 1$ and Assumption \ref{assumption: bounded-solution} holds with diameter $\cD < \infty$ and suppose for some $p>1$ and a fixed reference point
$x^\star\in\cX$, the noise satisfies
\begin{align*}
    \E[\|w_k\|^p\mid\cF_k]
    \leq A_p+B_p\|x_k-x^\star\|^p,
    \qquad \text{a.s.}, \quad \forall k\geq0.
\end{align*}
Then, almost surely, we have for every $k\geq0$ that
\begin{align*}
    \E[\|w_k\|^p\mid\cF_k]
    \leq A_{p,\cD}
    +B_{p,\cD}\inf_{x^\star\in\cX}\|x_k-x^\star\|^p,
\end{align*}
where $A_{p,\cD} = A_p+2^{p-1}B_p\cD^p, B_{p,\cD} = 2^{p-1}B_p$.
\end{claim}

\begin{proof}
Fix $k\geq0$ and work on the probability-one event on which
the assumed moment bound holds. For every $y^\star\in\cX$,
the triangle inequality and the definition of $\cD$ yield
\begin{align*}
    \|x_k-x^\star\|
    &\leq \|x_k-y^\star\|+\|y^\star-x^\star\| \leq \|x_k-y^\star\|+\cD.
\end{align*}
Using $(a+b)^p\leq2^{p-1}(a^p+b^p)$ for $a,b\geq0$, we obtain
\begin{align*}
    \E[\|w_k\|^p\mid\cF_k]
    &\leq A_p+B_p(\|x_k-y^\star\|+\cD)^p \\
    &\leq A_p+2^{p-1}B_p\cD^p
        +2^{p-1}B_p\|x_k-y^\star\|^p \\
    &= A_{p,\cD}+B_{p,\cD}\|x_k-y^\star\|^p.
\end{align*}
Since this inequality holds for every $y^\star\in\cX$
on the same event, taking the infimum gives
\begin{align*}
    \E[\|w_k\|^p\mid\cF_k]
    \leq A_{p,\cD}
    +B_{p,\cD}\inf_{y^\star\in\cX}\|x_k-y^\star\|^p.
\end{align*}
Finally, taking the countable intersection of these
probability-one events over $k\geq0$ proves that the bound
holds simultaneously for all $k$ almost surely.
\end{proof}

\subsection{Case 1: \texorpdfstring{$p \in (1,2]$}{1 < p <= 2}}
\begin{proof}
    First, let's suppose $x \neq x^\star$, then using the following inequality proved in Theorem 6.3 of \cite{rodomanov2020smoothness} for $p \in (1,2]$:
    \begin{align}
        \label{equation: sub_quadratic_drift}
        \|v + u\|^p \leq \|v\|^p + p \frac{\langle v, u \rangle}{\|v\|^{2-p}} + 2^{2-p}\|u\|^p
    \end{align}
    with $v = x_{k} - x^\star, u = \alpha_k(H(x_k) - x_k + w_k)$ (where $x^\star$ is any solution of the fixed-point equation), we have the estimation:
    \begin{align*}
        \|x_{k+1} - x^\star\|^p &\leq \|x_k - x^\star\|^p + p \alpha_k \frac{\langle x_k - x^\star, H(x_k) - x_k + w_k \rangle}{\|x_k - x^\star\|^{2-p}} + 2^{2-p}\alpha_k^p\|H(x_k) - x_k + w_k\|^p \\
        &\leq \|x_k - x^\star\|^p + p \alpha_k \frac{\langle x_k - x^\star, H(x_k) - x_k + w_k \rangle}{\|x_k - x^\star\|^{2-p}} \\
        &+ 2^{2-p}\alpha_k^p (\|H(x_k) - x^\star\| + \|x_k - x^\star\| + \|w_k\|)^p \\
        &\leq \|x_k - x^\star\|^p + p \alpha_k \frac{\langle x_k - x^\star, H(x_k) - x_k + w_k \rangle}{\|x_k - x^\star\|^{2-p}} \\
        &+ \underbrace{2^{2-p}3^{p-1}}_{C_{HT}/2} \alpha_k^p (2\|x_k - x^\star\|^p + \|w_k \|^p ) \qquad\qquad \text{ from Holder's inequality.}
     \end{align*}
     Next, we notice that by the non-expansive property of the operator $H$, we have:
     \begin{align*}
         &\|H(x_k) - x^\star\|^2 \leq \|x_k - x^\star\|^2 \text{ }\forall x^\star \in \cX \\
         \Rightarrow &\|H(x_k) - x_k\|^2 + 2 \langle H(x_k) - x_k, x_k - x^\star \rangle + \|x_k - x^\star\|^2 \leq  \|x_k - x^\star\|^2  \\
         \Rightarrow & \langle H(x_k) - x_k, x_k - x^\star \rangle \leq \frac{-\|H(x_k) - x_k\|^2}{2} \\
         \Rightarrow & \|x_k - x^\star\|^{p-2} \langle H(x_k) - x_k, x_k - x^\star \rangle \leq \frac{-\|x_k - x^\star\|^{p-2} \|H(x_k) - x_k\|^2}{2}. \\
     \end{align*}

     Thus, by taking conditional expectation , we get:

     \begin{align}
         &\E[\|x_{k+1} - x^\star\|^p | \cF_k] \leq \|x_k - x^\star\|^p + p \alpha_k \frac{\langle x_k - x^\star, H(x_k) - x_k \rangle}{\|x_k - x^\star\|^{2-p}} \notag \\
         &\qquad + C_{HT}/2 \times \alpha_k^p \br{2\|x_k - x^\star\|^p + \E[\|w_k \|^p | \cF_k]} \notag \\
         &\qquad \leq (1 + C_{HT}\alpha_k^p)\|x_k - x^\star\|^p + p \alpha_k \frac{\langle x_k - x^\star, H(x_k) - x_k \rangle}{\|x_k - x^\star\|^{2-p}} + C_{HT}\alpha_k^p \E[\|w_k\|^p | \cF_k] \notag \\
         \label{equation: p<2_nonexpansive_drift_before_inf}
        &\qquad \leq (1 + C_{HT} \alpha_k^p)\|x_k - x^\star\|^p - \frac{p}{2}\alpha_k \|x_k - x^\star\|^{p-2} \|H(x_k) - x_k\|^2 + C_{HT} \alpha_k^p \E[\|w_k\|^p | \cF_k].
     \end{align}
     Since \eqref{equation: p<2_nonexpansive_drift_before_inf} holds for any choice of $x^\star \in \cX$, taking $\inf$ over the choice of $x^\star \in \cX$ yields:
     \begin{align}
    \E\!\left[ \inf_{x^\star \in \cX}
        \|x_{k+1} - x^\star\|^p \,\big|\, \cF_k \right]
    &\leq (1 + C_{HT} \alpha_k^p)
        \inf_{x^\star \in \cX} \|x_k - x^\star\|^p \notag
        \\
        &\quad - \frac{p}{2}\alpha_k
          \sup_{x^\star \in \cX}
          \|x_k - x^\star\|^{p-2}
          \|H(x_k) - x_k\|^2 \, +\, C_{HT}\alpha_k^p
        \E[\|w_k\|^p | \cF_k] \notag \\
    &\leq (1 + C_{HT} \alpha_k^p + C_{HT} B_{p,\cD} \alpha_k^p)
        \inf_{x^\star \in \cX} \|x_k - x^\star\|^p
         \notag\\
    &\quad - \frac{p}{2}\alpha_k
          \sup_{x^\star \in \cX}
          \|x_k - x^\star\|^{p-2}
          \|H(x_k) - x_k\|^2 +\, C_{HT} A_{p,\cD} \alpha_k^p.
          \label{eqn: non-expansive-heavy-tailed-drift}
    \end{align}
    Now, consider $x_k=x^\star$, then $H(x_k)=x_k$, and hence \eqref{eqn: non-expansive-heavy-tailed-drift} also holds as
    \begin{align*}
        &\E\sqbr{\inf_{x^\star \in \cX} \|x_{k+1}-x^\star\|^p | \cF_k} \leq \E\sqbr{\|\alpha_k w_k\|^p | \cF_k} = \inf_{x^\star \in \cX} \|x_k - x^\star\|^p + \alpha_k^p \E\sqbr{\|w_k\|^p | \cF_k} \\
        &\leq \inf_{x^\star \in \cX} \|x_k - x^\star\|^p - \frac{p}{2}\alpha_k \sup_{x^\star \in \cX} \|x_k - x^\star\|^{p-2} \|H(x_k) - x_k\|^2 + C_{HT} A_{p,\cD} \alpha_k^p \\
        &\leq (1 + C_{HT} \alpha_k^p + C_{HT} B_{p,\cD} \alpha_k^p)
        \inf_{x^\star \in \cX} \|x_k - x^\star\|^p \\
        &- \frac{p}{2}\alpha_k \sup_{x^\star \in \cX} \|x_k - x^\star\|^{p-2} \|H(x_k) - x_k\|^2 +\, C_{HT} A_{p,\cD} \alpha_k^p,
    \end{align*}

     where $C_{HT} > 1$ for $p > 1$ and we consider $\sup_{x^\star \in \cX} \|x_k - x^\star\|^{p-2} \|H(x_k) - x_k\|^2 = 0$. Hence, applying the Supermartingale Convergence Theorem \ref{theorem: sct_theorem}, we have $\inf_{x^\star \in \cX}\|x_k - x^\star\|^p$ converges almost surely and $\sum \alpha_k \sup_{x^\star \in \cX} \|x_k - x^\star\|^{p-2} \|H(x_k) - x_k\|^2 < \infty$ almost surely.

     Now, we wish to show that $\inf_{x^\star \in \cX}\|x_k - x^\star\|^p$ converges almost surely to $0$. Suppose otherwise, then fix a sample path and suppose $\inf_{x^\star \in \cX} \|x_k  - x^\star\| \rightarrow \epsilon_1$ a.s for some $\epsilon_1 > 0$. This means for a sufficiently large $K_{\epsilon_1}$, we have $\epsilon_1/2 \leq \inf_{x^\star \in \cX} \|x_k  - x^\star\| \leq 2\epsilon_1$ and so $\sup_{x^\star \in \cX} \|x_k  - x^\star\|^{p-2} \geq \br{2\epsilon_1}^{p-2}$ for all $k \geq K_{\epsilon_1}$. However, Lemma \ref{lemma: H(x)-x bound when x is far from x*} then implies $\|H(x_k) - x_k\| \geq \epsilon_2$ for some $\epsilon_2 > 0$, which implies that
     \begin{align}
         \sum_{k=1}^\infty \alpha_k \sup_{x^\star \in \cX} \|x_k - x^\star \|^{p-2} \|H(x_k) - x_k \|^2 \geq \sum_{k=\lceil K_{\epsilon_1} \rceil}^\infty \alpha_k \br{2\epsilon_1}^{p-2}\epsilon_2^2 = \infty
     \end{align}
     from the step size condition \eqref{eqn: step-size-condition}. This contradicts the implication from Theorem \ref{theorem: sct_theorem}. Hence, we must have $\inf_{x^\star \in \cX} \norm{x_k-x^\star} \rightarrow 0$ almost surely.
\end{proof}

\subsection{Case 2: \texorpdfstring{$p \in (2,\infty)$}{p > 2}}
\begin{proof}
For $p > 2$, we further assume that $B_p = 0$, thus the noise becomes additive. In this case, the following inequality is established in Equation 14 in \citep{zaiwei-envelope}:

\begin{align}
\mathbb{E}\!\left[\inf_{x^\star \in \cX} \|x_{k+1} - x^\star\|_2^2 \,\middle|\, \mathcal{F}_k \right]
&\leq \inf_{x^\star \in \cX} \mathbb{E}\!\left[\|x_{k+1} - x^\star\|_2^2 \,\middle|\, \mathcal{F}_k \right] \notag \\
&\leq \inf_{x^\star \in \cX} \|x_k - x^\star\|_2^2 - (1 - \alpha_k)\alpha_k \|H(x_k) - x_k\|_2^2 + A_2 \alpha_k^2.
\label{equation: non_expansive_drift}
\end{align}
Next, we establish a bound on the step-wise distance to the set $\cX$:
\begin{claim}
\label{claim: non-expansive-gap-bound}
Almost surely, for all $k \geq 1$, the following inequality holds:
\begin{align}
\label{equation: non_expansive_gap_bounding}
\bigg|\inf_{x^\star \in \cX} |x_{k+1} - x^\star|
- \inf_{x^\star \in \cX} |x_k - x^\star| \bigg|
\leq 2\alpha_k\left(\inf_{x^\star \in \cX}|x_k - x^\star| + |w_k|\right).
\end{align}
\end{claim}

\begin{proof}
Indeed, this is true since for any $x^\star \in \cX$, we have:
\begin{align*}
    \|x_{k+1} - x_k\|
    &\le \alpha_k \left( \|H(x_k) - x_k\| + \|w_k\| \right) \notag\\
    &\le \alpha_k \left( \|H(x_k) - H(x^\star)\| + \|x_k - x^\star\| + \|w_k\| \right) \notag\\
    &\le 2\alpha_k \br{\|x_k - x^\star\| + \|w_k\|},
    \qquad \text{a.s.}
\end{align*}
Furthermore, note that from triangle inequality, we have:
\begin{align*}
    &\|x_{k+1} - x^\star\| \leq \|x_k - x^\star\| + \|x_{k+1} - x_k\| \text{ } \forall x^\star \in \cX \\
    \Rightarrow \inf_{x^\star \in \cX} &\|x_{k+1} - x^\star\| \leq \inf_{x^\star \in \cX} \|x_k - x^\star\| + \|x_{k+1} - x_k\|.
\end{align*}
Similarly, we have $\inf_{x^\star \in \cX} \|x_{k+1} - x^\star\| \geq \inf_{x^\star \in \cX} \|x_k - x^\star\| - \|x_{k+1} - x_k\|$, and so
\begin{align*}
    \bigg| \inf_{x^\star \in \cX} &\|x_{k+1} - x^\star\| - \inf_{x^\star \in \cX} \|x_k - x^\star\| \bigg| \leq \|x_{k+1} - x_k\|.
\end{align*}
Together this gives the inequality \eqref{equation: non_expansive_gap_bounding}, as desired.
\end{proof}

Next, in order to reuse our iteration projection technique, we need to prove that such a choice of Lyapunov function yields a "negative drift" strong enough to dominate higher order terms. Thus, we state and prove the following claim that is analogous to Claim \ref{claim: uk-negative-drift-alpha-k} in Theorem \ref{theorem: main-result}.

\begin{claim}
Consider the setting of Theorem \ref{theorem: nonexpansive-a.s.-convergence} and let $D > 0$ be an arbitrary constant. Then there exist constants $K \in \mathbb{N}$ and $b_1 > 0$ (both possibly depending on $D$) such that for all $k \geq K$, if
\begin{align*}
    D \;\leq\; \inf_{x^\star \in \cX} \|x_k - x^\star\| \;\leq\; 2D,
\end{align*}
then
\begin{align}
    \sqrt{\inf_{x^\star \in \cX} \|x_k - x^\star\|_2^2 - (1 - \alpha_k)\alpha_k \|H(x_k) - x_k\|_2^2 + A_2 \alpha_k^2} - \inf_{x^\star \in \cX} \|x_k - x^\star\|_2 \;\leq\; - b_1 \alpha_k.
\end{align}
\end{claim}

\begin{proof}
    Indeed, let $E$ be the set of $x$ such that $D \le \inf_{x^\star \in \cX} \|x - x^\star\| \le 2D$, since $\inf_{x^\star \in \cX} \|x - x^\star\|$ is continuous, set E is compact because from the fact that $\cX$ is bounded and closed. Thus, for all value $x \in E$, by Weierstrass theorem, there exists $\gamma_D > 0$ such that $\|H(x) - x\| > \gamma_D$.

    We claim that $b_1$ can be chosen as $\frac{\gamma_D^2}{8D} > 0$. With this choice, rearranging the target inequality requires us to show:
\begin{equation} \label{eq:target_ineq}
\sqrt{ \inf_{x^\star \in \cX} \|x_k - x^\star\|_2^2 - (1 - \alpha_k)\alpha_k \|H(x_k) - x_k\|_2^2 + A_2 \alpha_k^2}
\;\leq\; \inf_{x^\star \in \cX} \|x_k - x^\star\|_2 - \frac{\gamma_D^2}{8D}\alpha_k.
\end{equation}
Since $\alpha_k \to 0$ as $k \to \infty$ and by hypothesis $\inf_{x^\star \in \cX} \|x_k - x^\star\|_2 \geq D$, there exists an integer $K_1 \in \mathbb{N}$ such that for all $k \geq K_1$, the right-hand side is strictly positive:
\begin{align*}
    \inf_{x^\star \in \cX} \|x_k - x^\star\|_2 - \frac{\gamma_D^2}{8D}\alpha_k \;\geq\; D - \frac{\gamma_D^2}{8D}\alpha_k \;>\; 0.
\end{align*}
Because both sides of \eqref{eq:target_ineq} are non-negative for $k \geq K_1$, squaring both sides yields a sequence of strictly equivalent inequalities:
\begin{align}
    &\inf_{x^\star \in \cX} \|x_k - x^\star\|_2^2 - (1 - \alpha_k)\alpha_k \|H(x_k) - x_k\|_2^2 + A_2 \alpha_k^2 \nonumber\\
    &\hspace{2cm} \leq\; \left( \inf_{x^\star \in \cX} \|x_k - x^\star\|_2 - \frac{\gamma_D^2}{8D}\alpha_k\right)^2 \nonumber\\
    \Leftrightarrow \;& \inf_{x^\star \in \cX} \|x_k - x^\star\|_2^2 - (1 - \alpha_k)\alpha_k \|H(x_k) - x_k\|_2^2 + A_2 \alpha_k^2 \nonumber\\
    &\hspace{2cm} \leq\; \inf_{x^\star \in \cX} \|x_k - x^\star\|_2^2 - \frac{\gamma_D^2}{4D}\alpha_k \inf_{x^\star \in \cX} \|x_k - x^\star\|_2 + \frac{\gamma_D^4}{64D^2}\alpha_k^2 \nonumber\\
    \Leftrightarrow \;& -(1 - \alpha_k)\alpha_k \|H(x_k) - x_k\|_2^2 + A_2 \alpha_k^2 \nonumber\\
    &\hspace{2cm} \leq\; - \frac{\gamma_D^2}{4D}\alpha_k \inf_{x^\star \in \cX} \|x_k - x^\star\|_2 + \frac{\gamma_D^4}{64D^2}\alpha_k^2 \nonumber\\
    \Leftrightarrow \;& (1 - \alpha_k)\alpha_k \|H(x_k) - x_k\|_2^2 \nonumber\\
    &\hspace{2cm} \geq\; \alpha_k \frac{\gamma_D^2}{4D} \inf_{x^\star \in \cX} \|x_k - x^\star\|_2 + \alpha_k^2 \br{A_2 - \frac{\gamma_D^4}{64D^2}}. \label{eq:equiv_ineq}
\end{align}
To prove that \eqref{eq:equiv_ineq} holds, we bound both sides using our hypotheses. Since $x_k \in E$, we apply $\|H(x_k) - x_k\| \geq \gamma_D$ to lower-bound the left-hand side, and the assumption $\inf_{x^\star \in \cX} \|x_k - x^\star\|_2 \leq 2D$ to upper-bound the first term of the right-hand side. Therefore, it is sufficient to prove the following stronger inequality:
\begin{equation}
    \label{eq:suff_ineq}
    (1-\alpha_k)\alpha_k \gamma_D^2 \;\geq\; \frac{\gamma_D^2}{4D}\alpha_k (2D) - \frac{\gamma_D^4}{64D^2}\alpha_k^2 + A_2 \alpha_k^2.
\end{equation}
Simplifying the first term on the right-hand side yields:
\begin{equation}
    \label{eq:suff_ineq_simplified}
    (1-\alpha_k)\alpha_k \gamma_D^2 \;\geq\; \frac{\gamma_D^2}{2}\alpha_k + \alpha_k^2\left(A_2 - \frac{\gamma_D^4}{64D^2}\right).
\end{equation}
Since $\alpha_k > 0$, we divide \eqref{eq:suff_ineq_simplified} entirely by $\alpha_k$ to obtain the equivalent condition:
\begin{equation} \label{eq:final_ineq}
(1-\alpha_k)\gamma_D^2 \;\geq\; \frac{\gamma_D^2}{2} + \alpha_k\left(A_2 - \frac{\gamma_D^4}{64D^2}\right).
\end{equation}

Notice that as $k \to \infty$, the sequence $\alpha_k \to 0$. Taking the limit as $k \to \infty$ on both sides of \eqref{eq:final_ineq}, the left-hand side converges to $\gamma_D^2$, while the right-hand side converges to $\frac{\gamma_D^2}{2}$. Since $\gamma_D^2 > \frac{\gamma_D^2}{2}$ holds strictly, there exists an integer $K_2 \in \mathbb{N}$ such that for all $k \geq K_2$, the inequality \eqref{eq:final_ineq} is satisfied. Setting $K = \max\{K_1, K_2\}$ ensures that for all $k \geq K$, the sequence of equivalent inequalities holds true, satisfying our claim.
\end{proof}

Let $u_k = \inf_{x^\star \in \cX} \|x_k - x^\star\|_2$. We define
\begin{align}
z_k
&= \max(0,\min(2D,u_k)-D)
= \max(D,\min(2D,u_k)) - D
\in [0,D],
\end{align}
where $D>0$ is an arbitrary positive constant. We aim to show the following drift inequality:

\begin{lemma}
    \label{lemma: zk-drift-non-expansive}
    There exists a constant $b_4$ such that for all sufficiently large $k$, we have
    \begin{align}
        \label{equation: non_expansive_p_larger_rare_event_drift}
        \E[z_{k+1}^p \mid \cF_k] \le z_k^p + b_4 \alpha_k^p.
    \end{align}
\end{lemma}
This will allow us to apply the Supermartingale Convergence Theorem~\ref{theorem: sct_theorem} to show that $z_k$ converges almost surely.

\begin{proof}
Just as in the proof for Theorem \ref{theorem: main-result}, we consider the following cases:

\textbf{Case 1: $u_k \ge 2D$.}

In this case, $z_k = D$.
Since $z_{k+1} \le D$ almost surely, we immediately obtain
\begin{align}
\mathbb{E}[z_{k+1}^p \mid \mathcal{F}_k]
\le z_k^p .
\end{align}

\textbf{Case 2: $D \le u_k < 2D$.}

Here, $z_k = u_k - D$.
As established previously, we have $\big| \min(2D,u_{k+1}) - u_k \big| \le |u_{k+1} - u_k|$ almost surely.
Applying Lemma~\ref{lemma: p_larger_2_lp_regression_drift} with $x=z_k$ and $\Delta=\min(2D,u_{k+1})-u_k$, and utilizing the gap bound $|u_{k+1} - u_k| \le 2\alpha_k (u_k + \|w_k\|)$, we can bound the higher-order moments. Since $u_k < 2D$, we apply the moment bounds for $w_k$ to get:
\begin{align}
\mathbb{E}[z_{k+1}^p \mid \mathcal{F}_k]
&\le
z_k^p
+ \underbrace{
p z_k^{p-1}
\mathbb{E}[u_{k+1}-u_k \mid \mathcal{F}_k]
+ C_1 z_k^{p-2} \alpha_k^2
}_{T_{\mathrm{mid}}}
+ C_2 \alpha_k^p ,
\end{align}
for some constants $C_1, C_2 > 0$.

To bound the expected drift $\mathbb{E}[u_{k+1} - u_k \mid \mathcal{F}_k]$, we use Jensen's inequality (since the square root function is strictly concave) and the non-expansive bounds:
\begin{align}
\mathbb{E}[u_{k+1} \mid \mathcal{F}_k]
&= \mathbb{E}\left[\inf_{x^\star \in \cX} \|x_{k+1} - x^\star\|_2 \mid \mathcal{F}_k\right] \notag\\
&\le \sqrt{\mathbb{E}\left[\inf_{x^\star \in \cX} \|x_{k+1} - x^\star\|_2^2 \mid \mathcal{F}_k\right]} \notag\\
&\le \sqrt{\inf_{x^\star \in \cX} \|x_k - x^\star\|_2^2 - (1 - \alpha_k)\alpha_k \|H(x_k) - x_k\|_2^2 + A_2 \alpha_k^2} .
\end{align}
Using the bound for the compact region $D \le u_k < 2D$, there exists an index $K$ such that for all $k \ge K$, we have:
\begin{align}
\mathbb{E}[u_{k+1} \mid \mathcal{F}_k] \le u_k - \frac{\gamma_D^2}{8D}\alpha_k .
\end{align}
Letting $b_1 = \frac{\gamma_D^2}{8D}$, we obtain $\mathbb{E}[u_{k+1} - u_k \mid \mathcal{F}_k] \le -b_1 \alpha_k$.
Substituting this into our bound yields
\begin{align}
T_{\mathrm{mid}} \le -p b_1 \alpha_k z_k^{p-1} + C_1 z_k^{p-2} \alpha_k^2.
\end{align}
Choose $b_5 = \frac{C_1}{p b_1}$. Then, whenever $z_k > b_5 \alpha_k$, the term $T_{\mathrm{mid}} \le 0$.
Otherwise, if $z_k \le b_5 \alpha_k$, we obtain
\begin{align}
T_{\mathrm{mid}} \le C_1 z_k^{p-2} \alpha_k^2 \le C_1 b_5^{p-2} \alpha_k^p .
\end{align}
Absorbing the constants yields
\begin{align}
\mathbb{E}[z_{k+1}^p \mid \mathcal{F}_k]
\le z_k^p + b_{4,1} \alpha_k^p .
\end{align}

\textbf{Case 3: $u_k < D$.}

In this case, $z_k = 0$, and since $u_k < D$:
\begin{align}
    z_{k+1} = z_{k+1} - z_k
    \le |u_{k+1}-u_k|
    \le 2\alpha_k (u_k + \|w_k\|)
    \le 2\alpha_k (D + \|w_k\|);
    \qquad \text{a.s.}
\end{align}
where the second estimation follows from Claim \ref{claim: non-expansive-gap-bound}. Therefore, since $z_k = 0$, we have
\begin{align}
    \mathbb{E}[z_{k+1}^p \mid \mathcal{F}_k]
    &\le 2^p \alpha_k^p \mathbb{E}[(D+\|w_k\|)^p \mid \mathcal{F}_k] \le 2^{2p-1} \sqbr{D^p + A_p + B_p D^p}\alpha_k^p =
    z_k^p + b_{4,2} \alpha_k^p.
\end{align}
Taking $b_4 = \max \{b_{4,1}, b_{4,2}\}$, the proof is complete.
\end{proof}

From Lemma \ref{lemma: zk-drift-non-expansive} and applying the Supermartingale Convergence Theorem \ref{theorem: sct_theorem}, we conclude that $z_k$ converges almost surely. Finally, we show that $z_k$ converges almost surely to $0$. For this, we adopt the same argument as in the proof for Theorem \ref{theorem: main-result}, and only consider the existence of a value $p \geq q \geq 2$ such that $\sum \alpha_k^q < \infty$ and $\sum \alpha_k^{q-1} = \infty$ (from Claim \ref{claim: s(q) finite, s(q-1) infty}). Thus, we have:
\begin{align*}
    \E\!&\sqbr{\alpha_{k+1}^{q-2}\inf_{x^\star \in \cX} \|x_{k+1} - x^\star\|_2^2 \,\middle|\, \cF_k}
    &\leq \alpha_k^{q-2}\inf_{x^\star \in \cX} \|x_k - x^\star\|_2^2 - (1 - \alpha_k)\alpha_k^{q-1} \|H(x_k) - x_k\|_2^2 + A_2 \alpha_k^q.
\end{align*}
Applying Lemma \ref{lemma: H(x)-x bound when x is far from x*}, if $z_k$ does not converge to $0$ almost surely, then there exists a constant $\epsilon > 0$ such that $\inf_{x^\star \in \cX} \|x_k - x^\star\| \ge \epsilon_1$ for all $k \geq K$ sufficiently large, and thus $\|H(x_k) - x_k\| \ge \epsilon_2$ for all $k \geq K$, where $\epsilon_1 > 0$ is some constant. However, by the Supermartingale Convergence Theorem, we also have $\sum (1-\alpha_k)\alpha_k^{q-1} \|H(x_k) - x_k\|^2 < \infty$. Combining these two facts yields $\sum (1-\alpha_k)\alpha_k^{q-1} < \infty$, contradiction (since $q$ is chosen following \ref{claim: s(q) finite, s(q-1) infty}). Thus, we must have $z_k \to 0$ almost surely. The rest of the argument is the same as in the proof of Theorem \ref{theorem: main-result}. Hence, we must have $\inf_{x^\star \in \cX} \norm{x_k-x^\star}$ converges to $0$ almost surely.
\end{proof}

\begin{remark}
While the proof technique is very similar to the proof of Theorem \ref{theorem: main-result}, we can only obtain an almost sure guarantee under the additive noise assumption for the case $p \in (2,\infty)$. This is because of the lack of a strong negative drift on $\inf \norm{x_k-x^\star}$. While Lemma \ref{lemma: H(x)-x bound when x is far from x*} shows that we have some negative drift when $x_k$ is some distance away from $x^\star$, we do not have an explicit quantifiable bound on $\norm{H(x)-x}$. Therefore, we cannot show that the negative drift is proportional to $\norm{x_k-x^\star}$ to cancel out the multiplicative noise term as in the proof of Theorem \ref{theorem: main-result}. In Proposition \ref{proposition: non-expansive-multiplicative-counterexample}, we confirm that obtaining almost sure convergence for the case $p \in (2,\infty)$ under the multiplicative noise assumption is indeed impossible, and thus confirming our aforementioned intuition.

\end{remark}

\section{Proof details of other results}
\label{sec: proof-details}
\subsection{Proof for Proposition \ref{prop:fourth-drift}}
\label{sec: fourth-drift-proof}
\begin{proof}
Rewrite the recursion as
\begin{align}
x_{k+1} = (1-\alpha_k)x_k + \alpha_k w_k .
\end{align}
A direct binomial expansion gives
\begin{align}
    x_{k+1}^4 &= (1-\alpha_k)^4 x_k^4 + 4(1-\alpha_k)^3\alpha_k x_k^3 w_k \nonumber\\
    &\quad + 6(1-\alpha_k)^2\alpha_k^2 x_k^2 w_k^2
    + 4(1-\alpha_k)\alpha_k^3 x_k w_k^3
    + \alpha_k^4 w_k^4.
\end{align}
Taking conditional expectations and using
\begin{align*}
    \E[w_k \mid \mathcal F_k] &= 0, \\
    \E[|w_k|^p \mid \mathcal F_k] &< \infty, \qquad p=2,3,4,
\end{align*}
we obtain
\begin{align}
    \E[x_{k+1}^4 \mid \mathcal F_k] \le (1-\alpha_k)^4 x_k^4 + C_1 \alpha_k^2 x_k^2 + C_2 \alpha_k^3 |x_k| + C_3 \alpha_k^4.
\end{align}
Since $(1-\alpha_k)^4 \le (1-\alpha_k)$ for $\alpha_k \in [0,1]$, we apply Young's inequality $ ab \le \varepsilon a^2 + \frac{1}{4\varepsilon} b^2 $ to obtain
\begin{align}
    \alpha_k^2 x_k^2 = \alpha_k \cdot (\alpha_k x_k^2) \le \varepsilon \alpha_k x_k^4 + \frac{1}{4\varepsilon}\alpha_k^3,
\end{align}
and
\begin{align}
    \alpha_k^3 |x_k| = \alpha_k^2 \cdot (\alpha_k |x_k|) \le \alpha_k^2\!\left(
    \varepsilon' x_k^2 + \frac{1}{4\varepsilon'}\alpha_k^2
    \right) = \varepsilon' \alpha_k^2 x_k^2 + O(\alpha_k^4).
\end{align}
Substituting the second bound into the first (after collecting constants), and choosing $\varepsilon,\varepsilon'>0$ sufficiently small, yields
\begin{align}
    C_1 \alpha_k^2 x_k^2 + C_2 \alpha_k^3 |x_k|
    \le \tfrac12 \alpha_k x_k^4 + C \alpha_k^3.
\end{align}
Consequently,
\begin{align}
    \E[x_{k+1}^4 \mid \mathcal F_k]
    &\le x_k^4 - 4\alpha_k x_k^4 + O(\alpha_k^2) x_k^4
    + \tfrac12 \alpha_k x_k^4 + C' \alpha_k^3 \nonumber\\
    &\le x_k^4 - c\,\alpha_k x_k^4 + C'\alpha_k^3,
\end{align}
for some constants $c,C'>0$.
\end{proof}

\begin{remark}[On the order of the $\alpha_k^3$ term]

We believe that the $\alpha_k^3$ term cannot be eliminated if one performs a direct calculation. Indeed, consider the recursion
\begin{align}
    x_{k+1} = (1-\alpha_k)x_k + \alpha_k w_k,
\end{align}
where $\{w_k\}_{k \geq 0}$ is an i.i.d.\ sequence with $\E[w_k]=0$ and $\E[w_k^4]<\infty$, but $\E[w_k^3]\neq 0$. For example, let
\begin{align}
w_k =
\begin{cases}
2, & \text{with probability } \tfrac13,\\
-1, & \text{with probability } \tfrac23.
\end{cases}
\end{align}
Then $\E[w_k]=0$ while $\E[w_k^3]=2$. A direct computation gives
\begin{align*}
    \E[x_{k+1}^4 \mid \mathcal F_k]
    &= (1-\alpha_k)^4 x_k^4 + 6(1-\alpha_k)^2\alpha_k^2 x_k^2 \E[w_k^2 \mid \mathcal F_k] \\
    &\quad + 4(1-\alpha_k)\alpha_k^3 x_k \E[w_k^3 \mid \mathcal F_k] + \alpha_k^4 \E[w_k^4 \mid \mathcal F_k].
\end{align*}
Fixing $x_k \equiv x \neq 0$, the third-moment term contributes at order $\alpha_k^3$ and does not cancel in general. This indicates that, without further symmetry or moment conditions, an $\alpha_k^3$ remainder may naturally arise in fourth-moment drift bounds.
\end{remark}

\subsection{Proof for tightness of \texorpdfstring{$\xi > 1/p$}{xi > 1/p} (Theorem \ref{thm:tight})}
\label{ssec: impossibility-theorem-proof}
Fix $p\in[1,\infty)$, $\xi\in(0,1/p]$, $\alpha>0$, and $K\ge 1$. Take $H \equiv 0$ (a contraction with $\gamma=0$, fixed point $0$), so $x_{n+1}=(1-\alpha_n)x_n+\alpha_n w_n$. Let the filtration be $\mathcal F_n=\{x_0,w_0,\ldots,w_{n-1},x_n\}$. For each $n$, define independent mean-zero noises $w_n$ by
\begin{align*}
    w_n =
    \begin{cases}
        s_n & \text{w.p. } q_n\\
        -s_n & \text{w.p. } q_n\\
        0   & \text{w.p. } 1-2q_n
    \end{cases},\:
    s_n:=\frac{4}{\alpha}(n+K)^{\xi},\:
    q_n:=c\,(n+K)^{-\xi p},
\end{align*}
with any $c\in(0,1/2]$. Then $\mathbb E[w_n| \mathcal F_n]=0$ (MDS) and
\begin{align*}
    \mathbb E\left[|w_n|^p| \mathcal F_n\right]=2q_n s_n^{p}
=2c\left(\frac{4}{\alpha}\right)^p<\infty,
\end{align*}
implying that $\sup_n \mathbb E|w_n|^p<\infty$.

Let $I_n:=\{|w_n|=s_n\}$. The $I_n$ are independent with $\mathbb P(I_n)=2q_n=2c(n+K)^{-\xi p}$. Since $\xi p\le 1$, $\sum_n \mathbb P(I_n)=\infty$; by Borel--Cantelli (for independent events), $I_n$ occurs infinitely often a.s. On $I_n$, we have $\alpha_n s_n=\alpha(n+K)^{-\xi}\cdot \frac{4}{\alpha}(n+K)^{\xi}=4$, hence
\begin{align*}
    x_{n+1}-x_n=-\alpha_n x_n \pm 4.
\end{align*}
For every real $u$, at least one of $|4-u|$ or $|-4-u|$ is $\geq 4$; since the sign $\pm$ is an independent symmetric coin flip, we have
\begin{align*}
    \mathbb P \left(|x_{n+1}-x_n|\geq 4 |\ \mathcal F_n,\ I_n \right)\ge \frac12.
\end{align*}
Define $J_n:=I_n\cap\{|x_{n+1}-x_n|\geq 4\}$. Then $\sum_n q_n=\infty$ and
\begin{align*}
    \mathbb P(J_n| \mathcal F_n)\ \geq \frac12\,\mathbb P(I_n| \mathcal F_n)\ =\ q_n.
\end{align*}
By L\'evy's conditional Borel--Cantelli lemma, $J_n$ occurs infinitely often a.s. Thus $|x_{n+1}-x_n|\geq 4$ infinitely often a.s., so $(x_n)$ cannot converge.

\begin{remark}
    While it seems that Theorem \ref{thm:tight} disproves SLLN for $\xi = p = 1$, it is crucial to note that the noise construction in this proof is a martingale difference noise sequence, and thus does not preclude the guarantee in the i.i.d. noise setting.
\end{remark}

\subsection{Proof for \texorpdfstring{$p > 2$}{p > 2} non-expansive counterexample}
\label{ssec: counterexample-p-larger-2-non-expansive}

\textbf{Intuition}: Observe that for $p > 2$, it is possible to choose a step size sequence $\{\alpha_k\}$ such that $\sum_{k=1}^\infty \alpha_k^2 = \infty$. And thus, we aim to construct a sequence such that
\begin{align*}
    \norm{x_{k+1}-x^\star}^2 \geq \br{1+\Omega(\alpha_k^2)}\norm{x_{k}-x^\star}^2,
\end{align*}
so that as $k \to \infty$, we have
\begin{align*}
    \norm{x_{k+1}-x^\star}^2 \geq e^{\sum_{i=0}^k \Omega\br{\alpha_i^2}} \norm{x_0-x^\star}^2 \to \infty.
\end{align*}
Indeed, we have the following proof.
\begin{proof}
Let $\overline{\mathbb{B}} :=\{x\in\mathbb{R}^{2}:\|x\|_2\leq 1\}$ be the closed Euclidean unit ball, and let $P_{\overline{\mathbb{B}}}$ denote the Euclidean projection onto $\overline{\mathbb{B}}$.  Define
\begin{align}
    H(x) :=x-P_{\overline{\mathbb{B}}}(x) =
    \begin{cases}
        0, & \|x\|_2\leq 1,\\[1mm]
        \left(1-\dfrac{1}{\|x\|_2}\right)x,
        & \|x\|_2>1.
    \end{cases}
    \label{equation: counterexample-operator}
\end{align}
We first verify that $H$ is non-expansive.  The projection onto a nonempty
closed convex set is firmly non-expansive, so, for all $x,y\in\mathbb{R}^2$,
\begin{align}
    \left\langle
        P_{\overline{\mathbb{B}}}(x)
        -P_{\overline{\mathbb{B}}}(y),x-y
    \right\rangle
    \geq
    \left\|
        P_{\overline{\mathbb{B}}}(x)
        -P_{\overline{\mathbb{B}}}(y)
    \right\|_2^2.
    \label{equation: projection-firm-non-expansiveness}
\end{align}
Writing $P=P_{\overline{\mathbb{B}}}$ and expanding the square gives
\begin{align}
    \|H(x)-H(y)\|_2^2
    &=
    \|(x-y)-(P(x)-P(y))\|_2^2 \notag\\
    &=
    \|x-y\|_2^2+\|P(x)-P(y)\|_2^2
    -2\langle x-y,P(x)-P(y)\rangle \notag\\
    &\leq
    \|x-y\|_2^2-\|P(x)-P(y)\|_2^2
    \leq \|x-y\|_2^2.
\end{align}
Thus, $H$ is non-expansive.  Moreover, if $\|x\|_2\leq 1$, then
$H(x)=x$ implies $x=0$; if $\|x\|_2>1$, then
$H(x)=(1-\|x\|_2^{-1})x\neq x$.  Hence, we have the fixed-point set $\cX=\{0\}$. Let
\begin{align*}
    J:=
    \begin{pmatrix}
        0 & -1\\
        1 & 0
    \end{pmatrix},
\end{align*}
so that $\langle x,Jx\rangle=0$ for every $x\in\mathbb{R}^2$. Let $\{\varepsilon_k\}_{k\geq 1}$ be independent Rademacher random variables, independent of the deterministic initial point $x_0$, and set
\begin{align}
    \mathcal{F}_k
    :=\sigma(x_0,\varepsilon_1,\ldots,\varepsilon_k),
    \qquad
    w_k:=\sigma\varepsilon_{k+1}Jx_k,
    \label{equation: counterexample-noise}
\end{align}
where $\sigma>0$ is fixed.  Because $x_k$ is $\mathcal{F}_k$-measurable and $\varepsilon_{k+1}$ is independent of $\mathcal{F}_k$ with zero mean, we have $\E[w_k | \cF_k] = 0$ and
\begin{align}
    \mathbb{E}\!\left[\|w_k\|_2^p\mid\mathcal{F}_k\right]
    =\sigma^p\|x_k\|_2^p.
    \label{equation: counterexample-noise-moment}
\end{align}
Thus, the multiplicative $p$-th moment condition holds with
$A_p=1$ and $B_p=\sigma^p>0$. Now, consider
\begin{align}
    \frac{1}{p}<\xi<\frac{1}{2},
    \qquad
    \alpha_k=(k+K)^{-\xi},
    \label{equation: counterexample-step-size}
\end{align}
where the integer $K\geq 1$ will be chosen sufficiently large below.
This choice of $\xi$ guarantees
\begin{align}
    \sum_{k=0}^{\infty}\alpha_k=\infty,
    \qquad
    \sum_{k=0}^{\infty}\alpha_k^p<\infty.
\end{align}
At the same time, we have $2\xi<1$, and hence
\begin{align}
    \sum_{k=0}^{\infty}\alpha_k^2=\infty.
    \label{equation: counterexample-step-size-not-square-summable}
\end{align}
It remains to prove divergence. Let $u_k:=\|x_k\|_2$.  Whenever $u_k>1$, Equation
\eqref{equation: counterexample-operator} and the update step
\eqref{equation: main_iterate} give
\begin{align}
    x_{k+1} = \left(1-\frac{\alpha_k}{u_k}\right)x_k + \sigma\alpha_k\varepsilon_{k+1}Jx_k.
\end{align}
The two vectors on the right-hand side are orthogonal.  Therefore,
\begin{align}
    u_{k+1}^2
    =(u_k-\alpha_k)^2+\sigma^2\alpha_k^2u_k^2.
    \label{equation: counterexample-radius-recursion}
\end{align}
Set $c:=\sigma^2/6$.  Since $\xi<1/2$, we may choose $K$ sufficiently
large that, for every $k\geq 0$,
\begin{align}
    \sigma^2\alpha_k^2\leq 1,
    \qquad
    \frac{\alpha_k}{\alpha_{k+1}}\leq 1+c\alpha_k^2,
    \qquad
    \frac{6}{\sigma^2\alpha_0}>1.
    \label{equation: counterexample-choice-of-K}
\end{align}
Indeed, concavity of $t\mapsto t^\xi$ gives
\begin{align*}
    \frac{\alpha_k}{\alpha_{k+1}}
    =
    \left(1+\frac{1}{k+K}\right)^\xi
    \leq 1+\frac{\xi}{k+K},
\end{align*}
whereas $\frac{\xi}{k+K} \leq c(k+K)^{-2\xi}=c\alpha_k^2$ for every $k\geq 0$ once $K$ is sufficiently large, because $1-2\xi>0$.  The other two conditions in \eqref{equation: counterexample-choice-of-K} also hold after increasing
$K$, if necessary. Now, for $0\leq t\leq 1$, we have
\begin{align}
    \sqrt{1+t}-1
    =\frac{t}{\sqrt{1+t}+1}
    \geq \frac{t}{3}.
    \label{equation: counterexample-square-root-bound}
\end{align}
Using the reverse triangle inequality in
\eqref{equation: counterexample-radius-recursion}, followed by
\eqref{equation: counterexample-square-root-bound}, yields
\begin{align}
    u_{k+1} \geq u_k\sqrt{1+\sigma^2\alpha_k^2}-\alpha_k \geq u_k\left(1+\frac{\sigma^2}{3}\alpha_k^2\right)-\alpha_k.
    \label{equation: counterexample-radius-lower-bound}
\end{align}
Consequently, whenever
\begin{align}
    u_k\geq \frac{6}{\sigma^2\alpha_k},
    \label{equation: counterexample-radius-threshold}
\end{align}
we have
$\alpha_k\leq(\sigma^2/6)\alpha_k^2u_k$, and therefore
\begin{align}
    u_{k+1}
    \geq u_k\left(1+c\alpha_k^2\right).
    \label{equation: counterexample-radius-growth}
\end{align}
Choose $x_0$ such that $u_0=\|x_0\|_2\geq\frac{6}{\sigma^2\alpha_0}$, we prove by induction that
\begin{align}
    u_k\geq\frac{6}{\sigma^2\alpha_k}>1
    \qquad \forall k\geq 0.
    \label{equation: counterexample-threshold-induction}
\end{align}
Indeed, we have that the assertion holds at $k=0$.  If it holds at $k$, then \eqref{equation: counterexample-radius-growth} and \eqref{equation: counterexample-choice-of-K} imply
\begin{align*}
    u_{k+1}
    &\geq
    u_k(1+c\alpha_k^2) \geq
    \frac{6}{\sigma^2\alpha_k}
    \frac{\alpha_k}{\alpha_{k+1}} = \frac{6}{\sigma^2\alpha_{k+1}}.
\end{align*}
This proves \eqref{equation: counterexample-threshold-induction}.
Iterating \eqref{equation: counterexample-radius-growth} now gives
\begin{align}
    u_k
    \geq
    u_0\prod_{j=0}^{k-1}(1+c\alpha_j^2).
    \label{equation: counterexample-divergent-product}
\end{align}
For all sufficiently large $j$, $c\alpha_j^2\leq 1$, and
$\log(1+c\alpha_j^2)\geq c\alpha_j^2/2$.  Equation
\eqref{equation: counterexample-step-size-not-square-summable} therefore
implies
\begin{align}
    \sum_{j=0}^{\infty}\log(1+c\alpha_j^2)=\infty.
\end{align}
The product in \eqref{equation: counterexample-divergent-product} diverges, and hence $u_k=\|x_k\|_2\to\infty$.  Since the radius recursion is independent of the Rademacher realizations, this divergence holds on every sample path and, in particular, almost surely.
\end{proof}

\begin{remark}
At first, it seems surprising that we can obtain an almost-sure convergence guarantee with multiplicative noise for $p \in (1,2]$ but not for $p \in (2,\infty)$. But by investigating the drift inequality \eqref{equation: non_expansive_drift} with the multiplicative noise assumption, we have
\begin{align*}
    \E\sqbr{\inf_{x^\star \in \cX} \norm{x_{k+1}-x^\star}_2^2 | \cF_k} \leq \br{1+\Theta(\alpha_k^2)} \inf_{x^\star \in \cX} \norm{x_k-x^\star}_2^2 - (1-\alpha_k)\alpha_k \norm{H(x_k)-x_k}_2^2 + A_2 \alpha_k^2.
\end{align*}
Hence, if we construct a suitable operator $H$ and an appropriate noise sequence $w_k$ then we can obtain a sequence $x_k$ so that $\E\sqbr{\inf_{x^\star \in \cX} \norm{x_{k+1}-x^\star}_2^2 | \cF_k} \geq \br{1+\Omega(\alpha_k^2)} \inf_{x^\star \in \cX} \norm{x_k-x^\star}_2^2$, which will diverge as $\sum_{k=1}^\infty \alpha_k^2$ diverges. Compared to the case where Assumption \ref{assumption: general-drift-condition} is satisfied, the negative drift term $\norm{H(x_k)-x_k}$ can be arbitrarily small and is not strong enough to cancel the effect from the multiplicative noise sequence.
\end{remark}

\end{document}